\documentclass{article}
\usepackage{amsmath}
\usepackage{amsthm}
\usepackage{amssymb}

\theoremstyle{definition}
\newtheorem{defi}{Definition}

\theoremstyle{plain}
\newtheorem{thm}[defi]{Theorem}
\newtheorem{cor}[defi]{Corollary}
\newtheorem{lem}[defi]{Lemma}
\newtheorem{mainlem}[defi]{Main Lemma}
\newtheorem{prelem}[defi]{Preliminary Lemma}
\newtheorem{prop}[defi]{Proposition}
\newtheorem{ques}[defi]{Question}
\newtheorem{obs}[defi]{Observation}
\newtheorem{ex}[defi]{Example}

\newcounter{enuroman}
\renewcommand{\theenuroman}{\roman{enuroman}}
\newenvironment{romanenumerate}{\begin{list}{\rm (\theenuroman)}{\usecounter{enuroman}
    \setlength{\labelwidth}{1cm}}}
   {\end{list}}

\newcommand{\forces}{\Vdash}
\newcommand{\re}{{\upharpoonright}}

\newcommand{\ka}{\kappa}
\newcommand{\al}{\alpha}
\newcommand{\be}{\beta}
\newcommand{\ga}{\gamma}

\newcommand{\A}{{\cal A}}
\newcommand{\B}{{\cal B}}
\newcommand{\C}{{\cal C}}
\newcommand{\D}{{\cal D}}

\newcommand{\F}{{\cal F}}
\newcommand{\calF}{{\cal F}}
\newcommand{\G}{{\cal G}}

\newcommand{\I}{{\cal I}}
\newcommand{\J}{{\cal J}}

\newcommand{\M}{{\cal M}}
\newcommand{\N}{{\cal N}}
\newcommand{\NWD}{{\cal{NWD}}}
\renewcommand{\P}{{\cal P}}

\newcommand{\U}{{\cal U}}

\newcommand{\X}{{\cal X}}
\newcommand{\Y}{{\cal Y}}

\renewcommand{\AA}{{\mathbb A}}
\newcommand{\BB}{{\mathbb B}}
\newcommand{\CC}{{\mathbb C}}
\newcommand{\DD}{{\mathbb D}}

\newcommand{\MM}{{\mathbb M}}
\newcommand{\MI}{{\mathbb{MI}}}

\newcommand{\PP}{{\mathbb P}}
\newcommand{\QQ}{{\mathbb Q}}

\renewcommand{\SS}{{\mathbb S}}

\newcommand{\LOCforce}{{\mathbb{LOC}}}
\newcommand{\PLOCforce}{{\mathbb{PLOC}}}

\newcommand{\bb}{{\mathfrak b}}

\newcommand{\cc}{{\mathfrak c}}
\newcommand{\dd}{{\mathfrak d}}

\newcommand{\uu}{{\mathfrak u}}

\newcommand{\xx}{{\mathfrak x}}
\newcommand{\yy}{{\mathfrak y}}

\newcommand{\Cov}{{\mathsf{Cov}}}

\newcommand{\Cof}{{\mathsf{Cof}}}
\newcommand{\Loc}{{\mathsf{Loc}}}
\newcommand{\LOC}{{\mathsf{LOC}}}

\newcommand{\add}{{\mathsf{add}}}
\newcommand{\cov}{{\mathsf{cov}}}
\newcommand{\non}{{\mathsf{non}}}
\newcommand{\cof}{{\mathsf{cof}}}

\newcommand{\Fn}{{\mathsf{Fn}}}

\newcommand{\pLoc}{{\mathsf{pLoc}}}
\newcommand{\pLOC}{{\mathsf{pLOC}}}
\newcommand{\p}{{\mathsf{p}}}

\newcommand{\stem}{{\mathrm{stem}}}
\newcommand{\Split}{{\mathrm{split}}}

\newcommand{\ran}{{\mathrm{ran}}}

\newcommand{\dom}{{\mathrm{dom}}}

\newcommand{\supp}{{\mathrm{supp}}}

\renewcommand{\succ}{{\mathrm{succ}}}

\newcommand{\cf}{{\mathrm{cf}}}

\newcommand{\Match}{{\mathrm{Match}}}
\newcommand{\id}{{\mathrm{id}}}
\newcommand{\CR}{{\mathrm{CR}}}
\newcommand{\IP}{{\mathrm{IP}}}

\newcommand{\one}{{\mathbf{1}}}

\newcommand{\sub}{\subseteq}
\newcommand{\sem}{\setminus}
\newcommand{\twoom}{2^\omega}

\newcommand{\omom}{\omega^\omega}

\newcommand{\ha}{\,{}\hat{}\,}
\newcommand{\la}{\langle}
\newcommand{\ra}{\rangle}

\title{Cicho\'{n}'s Diagram for uncountable cardinals}

\author{J\"org Brendle\thanks{Partially supported by Grants-in-Aid for Scientific Research
   (C) 24540126 and (C) 15K04977, Japan Society for the Promotion of Science, 
   by JSPS and FWF under the Japan-Austria Research Cooperative Program 
   {\em New developments regarding forcing in set theory}, 
   by the IMS program {\em Sets and computations}, National University of Singapore (April 2015), and by
   the Isaac Newton Institute for Mathematical Sciences in the programme {\em Mathematical, Foundational and 
   Computational Aspects of the Higher Infinite} (HIF) funded by EPSRC grant EP/K032208/1.}  \\
   Graduate School of System Informatics \\
   Kobe University \\
   Rokko-dai 1-1, Nada-ku \\
   Kobe 657-8501, Japan \\
   email: {\sf brendle@kobe-u.ac.jp} \\  \\
   and \\ \\
   Andrew Brooke-Taylor\thanks{Supported at the commencement of this research by a JSPS Postdoctoral Fellowship for Foreign Researchers and JSPS Grant-in-Aid no. 23 01765.  Currently supported by EPSRC Early Career Fellowship reference EP/K035703/1.  Some of this work was undertaken as a Visiting Fellow at the Isaac Newton Institute for Mathematical Sciences in the programme {\em Mathematical, Foundational and Computational Aspects of the Higher Infinite} (HIF) funded by EPSRC grant EP/K032208/1.} \\
   School of Mathematics \\
   University of Leeds \\
   Mathematics and Environment Building (84A) \\
   LS2 9JT  Leeds, UK \\
   email: {\sf  A.D.Brooke-Taylor@leeds.ac.uk} \\ \\
   and \\ \\
   Sy-David Friedman\thanks{Partially supported by the FWF through Project P25748 and
   by the IMS program {\em Sets and computations}, National University of Singapore (April 2015).} \\
   Kurt G\"odel Research Center \\
   Universit\"at Wien \\
   W\"ahringer Strasse 25 \\
   1090 Vienna, Austria \\ 
   email: {\sf sdf@logic.univie.ac.at} \\ \\
   and \\ \\
   Diana Montoya\thanks{Partially supported by the FWF through Project P25748 and
   by the IMS program {\em Sets and computations}, National University of Singapore (April 2015).} \\
   Kurt G\"odel Research Center \\
   Universit\"at Wien \\
   W\"ahringer Strasse 25 \\
   1090 Vienna, Austria \\
   email: {\sf dcmontoyaa@gmail.com}
}

\begin{document}
\maketitle

\begin{abstract}
\noindent We develop a version of Cicho\'n's diagram for cardinal invariants on the generalized Cantor space
$2^\kappa$ or the generalized Baire space $\kappa^\kappa$ where $\kappa$ is an uncountable regular cardinal.
For strongly inaccessible $\kappa$, many of the ZFC-results about the order relationship of the cardinal
invariants which hold for $\omega$ generalize; for example we obtain a natural generalization of the
Bartoszy\'nski-Raisonnier-Stern Theorem. We also prove a number of independence results, both with 
$<\kappa$-support iterations and $\kappa$-support iterations and products, showing that we consistently have strict
inequality between some of the cardinal invariants.
\end{abstract}



\section{Introduction}

Cardinal invariants of the continuum are cardinal numbers which are usually defined as the minimal size of a set of reals
with a certain property -- and thus describe the combinatorial structure of the real line -- and which typically take values between
the first uncountable cardinal $\aleph_1$ and the cardinality of the continuum $\cc$. Some of the most important cardinal invariants 
characterize the structure of the $\sigma$-ideals $\M$ and $\N$ of meager and Lebesgue measure zero (null) sets, respectively. The
order relationship between these cardinals, as well as the closely related bounding and dominating numbers $\bb$ and $\dd$, was intensively 
investigated in the 1980's and is usually displayed in {\em Cicho\'n's diagram}. (See Section 2 for definitions and for the diagram.)
The deepest result in this context is the Bartoszy\'nski-Raisonnier-Stern Theorem 
(see~\cite{Ba84}, \cite{RS85}, or~\cite[Theorem 2.3.1]{BJ95}) which says
that if the union of any family of $\kappa$ many null sets is null, then the union of any family of $\kappa$ many meager sets is
meager. In symbols, this is $\add (\N) \leq \add (\M)$, and the same proof yields the dual inequality $\cof (\M) \leq \cof (\N)$. Cicho\'n's
diagram is complete in the sense that any assignment of the cardinals $\aleph_1$ and $\aleph_2$ which does not contradict the
diagram is consistent with ZFC, the axioms of set theory~\cite[Sections 7.5 and 7.6]{BJ95}.

Starting with the 1990's, variations on Cicho\'n's diagram have been investigated. For example, Pawlikowski and Rec{\l}aw
(see~\cite{PR95} or~\cite[Theorem 3.11]{Ba10}) established that the Galois-Tukey morphisms between triples underlying the proofs of the inequalities
between cardinal invariants can all be taken to be continuous maps, thus obtaining a parametrized version of Cicho\'n's diagram
which also makes sense in the context of the continuum hypothesis CH. (See the end of Section 2 for details on
Galois-Tukey connections.) Both the classical Cicho\'n diagram and a Cicho\'n diagram for small
sets of reals are consequences of this parametrized
Cicho\'n diagram. More recently, Ng, Nies, and the first two authors of the present paper~\cite{BBNN15} developed a Cicho\'n diagram
for highness properties in the Turing degrees. 

In this paper, we attempt to generalize Cicho\'n's diagram in another direction: instead of looking at the Cantor space $\twoom$
or the Baire space $\omom$, we consider the {\em generalized Cantor space $2^\kappa$} or the {\em generalized Baire space $\kappa^\kappa$},
where $\kappa$ is a regular uncountable cardinal. Most of the combinatorial cardinal invariants can easily be redefined in this
context. Moreover, the meager ideal $\M$ on $\twoom$ has a natural analogue $\M_\kappa$ on $2^\kappa$ if we equip $2^\kappa$
with the $< \kappa$-box topology and call a subset of $2^\kappa$ {\em $\kappa$-meager} if it is a $\kappa$-union of nowhere
dense sets in this topology. It is unclear, however, how the null ideal $\N$ should be generalized to $2^\kappa$. Fortunately, most of the
cardinals in the classical Cicho\'n diagram have rather simple characterizations 
in purely combinatorial terms~\cite[Chapter 2]{BJ95}, and to obtain a version
of Cicho\'n's diagram for uncountable regular $\kappa$ containing analogues of at least some of the cardinals related to Lebesgue
measure, we generalize these combinatorial characterizations. This allows us to reprove, for example, a version of the
Bartoszy\'nski-Raisonnier-Stern Theorem which holds for strongly inaccessible $\kappa$ (see Theorem~\ref{BRT-thm} and
Corollary~\ref{BRT-cor} in Section 3).

It turns out that for extending some of the inequalities in Cicho\'n's diagram -- and the equalities establishing the combinatorial
characterizations of some of the cardinals -- additional assumptions on $\kappa$ are necessary. First of all, if $2^{< \kappa} > \kappa$,
then some of the cardinals are equal to $\kappa^+$ while $2^{< \kappa}$ is a lower bound for others so that the diagram is somewhat
degenerate. For this reason, we mainly focus on the case $2^{< \kappa} = \kappa$. But even then it is unclear whether all
inequalities generalize; for example, we do not know whether $\add (\M_\kappa) \leq \bb_\kappa$ holds for successor $\kappa$
(Question~\ref{add-cof-ques}). Fortunately, if $\kappa$ additionally is inaccessible, then a close analogue of Cicho\'n's diagram
can be redrawn and, furthermore, combinatorial characterizations of some of the cardinals also generalize.
(See the end of Section 3 for the diagram.)

A much harder problem is the generalization of independence results about the order-relationship between the cardinals.
For cardinal invariants $\xx$ and $\yy$ describing $\twoom$ (or $\omega^\omega$), the consistency of $\xx < \yy$ typically is
shown either by a long finite support iteration (fsi) of ccc forcing over a model of CH, making $\yy = \cc$ large while preserving 
$\xx = \aleph_1$, by a short fsi of ccc forcing over a model of MA $+ \neg$ CH, preserving $\yy = \cc$ and making $\xx$ small,
or by an $\aleph_2$-stage countable support iteration (csi) of proper forcing over a model of CH, making $\yy = \cc = \aleph_2$
and preserving $\xx = \aleph_1$. Assuming $2^{<\kappa} = \kappa$, these methods naturally generalize to $2^\kappa$ (or
$\kappa^\kappa$), the former two to $< \kappa$-support iteration of $\kappa^+$-cc and $<\kappa$-closed forcing, and, under
$2^\kappa = \kappa^+$, the latter to $\kappa$-support iteration of $\kappa^{++}$-cc, $<\kappa$-closed forcing preserving
$\kappa^+$. The main obstacle is that, unlike for $\kappa = \omega$, in neither case do we have {\em preservation results}~\cite[Chapter 6]{BJ95},
proved in an iterative fashion, and saying that $\xx$ stays small (for the first and third type of models) or that $\yy$ stays large
(for the second type of models). The problem is that when attempting to generalize the inductive proofs of such 
preservation results, we run into trouble at limit stages of cofinality $< \kappa$, a case that does not appear for $\kappa =\omega$.
A natural approach for getting around this problem is to show directly that the whole iteration has the necessary preservation
property and, indeed, this sometimes can be done, both for $<\kappa$-support constructions (Theorem~\ref{total-partial-thm} in 
Section 4 where we use the appropriate generalization of centeredness) and for $\kappa$-support constructions
(Main Lemma~\ref{prod-Sacksprop} and Theorem~\ref{Sacks-model} in Section 5 where we show directly that products
(and iterations) of Sacks forcing satisfy the appropriate generalization of the Sacks property). A further problem is that
in some cases (e.g., generalized Hechler forcing, Subsection 4.2), we even don't know whether the single-step
forcing has the property needed for preservation.

As a consequence, we are still far from knowing whether our diagram (for strongly inaccessible $\kappa$) is complete
in the sense that it shows all ZFC-results. Another interesting question we do not consider in our work is the problem of global
consistency results, see e.g.~\cite{CumSh} or~\cite{FT09}.

The paper is organized as follows. In Section 2, we introduce the basic cardinal invariants, present the classical Cicho\'n diagram
and give an outline of the theory of Galois-Tukey connections. In Section 3, we prove all ZFC-results leading to our
version of Cicho\'n's diagram for strongly inaccessible $\kappa$. We try to work with the weakest assumption in each case,
establishing some results for arbitrary regular uncountable $\kappa$ and some for such $\kappa$ with $2^{<\kappa} = \kappa$.
We also include a brief discussion of what happens if $2^{< \kappa} > \kappa$. Section 4 is about iterations with support of
size $< \kappa$ of $\kappa^+$-cc forcing. We discuss the effect of generalizations of Cohen forcing, Hechler forcing, and
localization forcing on the diagram and also present some models for the degenerate case $2^{< \kappa} > \kappa$
(see in particular Theorem~\ref{layers} which answers a question of Matet and Shelah~\cite{MSta}). The final Section 5 deals with 
iterations and products with support of size $\kappa$. We mainly investigate how generalized Sacks forcing and generalized 
Miller forcing change the cardinal invariants.



\section{Preliminaries}

Let $\kappa$ be an uncountable regular cardinal. Endow the space of functions $2^\kappa$ with the topology generated by the sets of the form $[s] = \{ f \in 2^\kappa : f \supseteq s\} $ for $s \in 2^{<\kappa}$.

We define the {\em generalized $\kappa$-meager sets} in $2^\ka$ to be $\kappa$-unions of nowhere dense sets with respect to this topology. Recall that $A \subseteq 2^\ka$ is nowhere dense if for every $s \in 2^{<\ka}$ there exists $t \supseteq s$ such that $[t] \cap A = \emptyset$. It is well known that the Baire category theorem can be lifted to the uncountable case, i.e. the intersection of $\ka$ many open dense sets is dense (see \cite{FrHyKu}). Let $\NWD_\kappa$ denote the ideal of nowhere dense sets and $\M_\kappa$, the $\kappa$-ideal of $\kappa$-meager sets. 
Here an ideal on a set is a {\em $\kappa$-ideal} if it is closed under unions of size $\kappa$.
Our goal is to generalize some of the cardinal invariants in Cicho\'{n}'s Diagram.

\begin{defi}
If $f, g$ are functions in $\kappa^\kappa$, we say that $g$ \emph{eventually dominates} $f$, and write $f <^{\ast} g$, if there exists an $\alpha<\kappa$ such that $f(\beta) < g(\beta)$ holds for all $\beta> \alpha$.
\end{defi}

\begin{defi}
Let $\mathcal{F}$ be a family of functions from $\kappa$ to $\kappa$. 
\begin{itemize}
\item $\mathcal{F}$ is {\em dominating} if for all $g \in \kappa^\kappa$, there exists an $f \in \mathcal{F}$ such that $g <^{\ast} f$.
\item $\mathcal{F}$ is {\em unbounded} if for all $g \in \kappa^\kappa$, there exists an $f \in \mathcal{F}$ such that $f \nless^{\ast} g$.
\end{itemize}
\end{defi}

\begin{defi}[The unbounding and dominating numbers, $\mathfrak{b}_\kappa$ and $\mathfrak{d}_\kappa$]\hfill
\begin{itemize}
\item $\mathfrak{b}_\kappa= \min\{ \lvert \mathcal{F} \lvert: \mathcal{F}$ is an unbounded family of functions in $\kappa^\kappa \}$.
\item $\mathfrak{d}_\kappa= \min\{ \lvert \mathcal{F} \lvert: \mathcal{F}$ is a dominating family of functions in $\kappa^\kappa \}$.
\end{itemize}
\end{defi}

When we refer to the cardinal invariants above in the case $\kappa=\omega$, we will just write $\mathfrak{b}$ and $\mathfrak{d}$.
It is well-known (see~\cite[Lemma 6]{CumSh}) that $\kappa^+ \leq \mathfrak{b}_\kappa = cf(\mathfrak{b}_\kappa) \leq cf(\mathfrak{d}_\kappa)
\leq \mathfrak{d}_\kappa \leq 2^\kappa$.

Say an ideal $\I$ on a set $X$ is {\em nontrivial} if $X\notin \I$ and $\bigcup \I = X$.

\begin{defi}[Cardinal invariants associated to an ideal] 
Let $\mathcal{I}$ be a nontrivial $\ka$-ideal on a set $X$. We define:
\begin{itemize} 
\item The {\em additivity number}: 
$\add(\mathcal{I}) = \min\{ \lvert \mathcal{J} \lvert: \mathcal{J} \subseteq \mathcal{I}\text{ and }\bigcup \mathcal{J} \notin \mathcal{I} \}.$
\item The {\em covering number}:
$ \cov(\mathcal{I})= \min \{\lvert \mathcal{J} \lvert: \mathcal{J} \subseteq \mathcal{I}\text{ and } \bigcup \mathcal{J} = X \}.$
\item The {\em cofinality number}: \\
$\cof(\mathcal{I})= \min \{\lvert \mathcal{J} \lvert: \mathcal{J} \subseteq \mathcal{I}\text{ and for all } I\in \mathcal{I} \text{ there is a }  J \in \mathcal{J} \text{ with } I \subseteq J\}.$
\item The {\em uniformity number}:
$\non(\mathcal{I})= \min \{\lvert Y\lvert: Y \subset X \text{ and } Y \notin \mathcal{I} \}.$
\end{itemize}
\end{defi}

Provable ZFC inequalities in the case $\ka=\omega$ and $\mathcal{I}$ being the $\sigma$-ideals $\M$ and $\N$ of meager and null sets of the Cantor space $2^\omega$ can be summarized in the well known Cicho\'{n} Diagram. Here an arrow ($\rightarrow$) means $\leq$.

\begin{figure}[ht]
\begin{center}
\setlength{\unitlength}{0.2000mm}
\begin{picture}(630.0000,180.0000)(0,10)
\thinlines
\put(535,180){\vector(1,0){50}}
\put(415,180){\vector(1,0){50}}
\put(175,180){\vector(1,0){50}}
\put(295,180){\vector(1,0){50}}
\put(295,100){\vector(1,0){50}}
\put(415,20){\vector(1,0){50}}
\put(295,20){\vector(1,0){50}}
\put(175,20){\vector(1,0){50}}
\put(55,20){\vector(1,0){50}}
\put(140,30){\vector(0,1){140}}
\put(500,30){\vector(0,1){140}}
\put(380,30){\vector(0,1){60}}
\put(380,110){\vector(0,1){60}}
\put(260,110){\vector(0,1){60}}
\put(260,30){\vector(0,1){60}}
\put(470,170){\makebox(60,20){$\cof(\mathcal{N})$}}
\put(470,10){\makebox(60,20){$\non(\mathcal{N})$}}
\put(350,170){\makebox(60,20){$\cof(\mathcal{M})$}}
\put(350,90){\makebox(60,20){$\mathfrak{d}$}}
\put(350,10){\makebox(60,20){$\cov(\mathcal{M})$}}
\put(230,10){\makebox(60,20){$\add(\mathcal{M})$}}
\put(230,90){\makebox(60,20){$\mathfrak{b}$}}
\put(230,170){\makebox(60,20){$\non(\mathcal{M})$}}
\put(110,170){\makebox(60,20){$\cov(\mathcal{N})$}}
\put(110,10){\makebox(60,20){$\add(\mathcal{N})$}}
\put(10,10){\makebox(40,20){$\aleph_1$}}
\put(590,170){\makebox(40,20){$\mathfrak{c}$}}
\end{picture}
\end{center}
\caption{Cicho\'n's diagram}
\label{ODiag}
\end{figure}
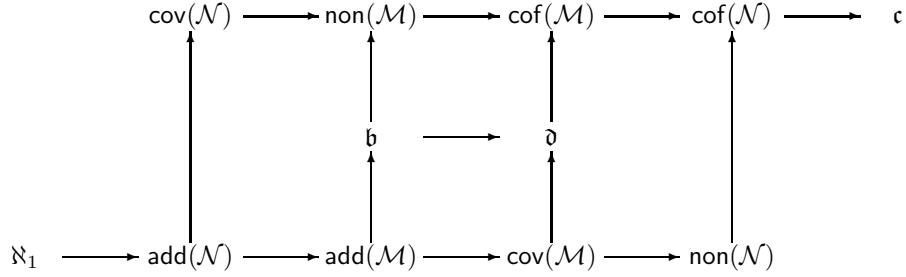

We will recall and use the main results on Galois-Tukey connections (see \cite{Bl10}), to characterize the cardinal invariants defined above.

\begin{defi}
Let $\mathbb{A}= (A_-, A_+, A)$ where $A_-$ and $A_+$ are two sets and $A$ is a binary relation on $A_- \times A_+$. We define the {\em norm} of the triple $\mathbb{A}$, $\| \mathbb{A} \|$, as the smallest cardinality of any subset $Y$ of $A_+$ such that every $x \in A_-$ is related by $A$ to at least one element $y \in Y$.
We also define the {\em dual} of $\mathbb{A}$, $\mathbb{A}^\bot= (A_+, A_-, \neg \check{A}) $ where $(x,y) \in \neg \check{A}$ if and only if $(y,x) \notin A$.  
\end{defi}

Then the cardinal invariants can be seen as norms of some specific triples:

\begin{itemize}
\item If $\mathcal{D}_\kappa= (\kappa^\kappa, \kappa^\kappa, \leq^\ast)$, then $\|\mathcal{D}_\kappa\|= \mathfrak{d}_\kappa$ and $\|\mathcal{D}^\perp_\kappa \|= \mathfrak{b}_\kappa$.
\item If $\Cov(\mathcal{M_\kappa})= (2^\kappa,\mathcal{M}_\kappa,\in)$, then $\| \Cov(\mathcal{M}_\kappa)\|= \cov(\mathcal{M}_\kappa)$ and $\| \Cov(\mathcal{M}_\kappa)^\perp \|$ \\ $= \non(\mathcal{M}_\kappa)$.  
\item If $\Cof(\mathcal{M_\kappa})= (\mathcal{M}_\kappa,\mathcal{M}_\kappa,\subseteq)$, then $\| \Cof(\mathcal{M}_\kappa)\|= \cof(\mathcal{M}_\kappa)$ and $\| \Cof(\mathcal{M}_\kappa)^\perp \|$\\ $= \add(\mathcal{M}_\kappa)$.
\end{itemize}

\begin{defi}  \label{morph-def}
A morphism from $\mathbb{A} = (A_-, A_+, A)$ to  $\mathbb{B}= (B_-, B_+, B)$  is a pair $\Phi= (\Phi_-, \Phi_+)$ of maps satisfying:
\begin{itemize}
\item $\Phi_-: B_- \to A_-$,
\item $\Phi_+: A_+ \to B_+$,
\item for all $b \in B_-$ and $a \in A_+$, if $\Phi_-(b)A a$ then $b B \Phi_+(a)$.\end{itemize}
We write $\mathbb{A} \preceq \mathbb{B}$ if there is a morphism from $\mathbb{A}$ to $\mathbb{B}$. If
$\mathbb{A} \preceq \mathbb{B}$ and $\mathbb{B} \preceq \mathbb{A}$ we write $\mathbb{A} \equiv \mathbb{B}$.
\end{defi}

\begin{obs}
$\mathbb{A} \preceq \mathbb{B}$, then $\| \AA \| \geq \| \BB\|$ and $\| \AA^\bot \| \leq \| \BB ^\bot\|$.
\end{obs}

\begin{ex}
There are morphisms $\Phi: \Cof(\mathcal{M}_\kappa) \preceq \Cov(\mathcal{M}_\kappa)$ and $\Psi: \Cof(\mathcal{M}_\kappa) \preceq \Cov(\mathcal{M}_\kappa)^\bot$ given by $\Phi=(S, \id)$ and $\Psi =(\id, N)$ where $S(x) = \{x\}$ and for $M \in \mathcal{M}_\kappa$, $N(M)$ is some element of $2^\kappa \setminus M$.
\end{ex}

\begin{cor}
$\add(\mathcal{M}_\kappa) \leq \cov(\mathcal{M}_\kappa) \leq \cof(\mathcal{M}_\kappa)$ and
$\add(\mathcal{M}_\kappa) \leq \non(\mathcal{M}_\kappa) \leq \cof(\mathcal{M}_\kappa)$. 
\end{cor}



\section{$ZFC$-results}

\subsection{Cardinal invariants of the meager ideal}

In this subsection, we generalize results about the central part of Cicho\'n's diagram, namely about the cardinals
related to the meager ideal as well as the bounding and dominating numbers, to uncountable regular $\kappa$.
While our main interest lies in the case $\kappa$ is strongly inaccessible, we try to be as general as possible,
also deal with successor $\kappa$ satisfying $2^{< \kappa} = \kappa$ -- and even with the somewhat degenerate
case when $2^{< \kappa} > \kappa$ (both for successors and (weakly) inaccessibles). In each of these cases
a lot is known, but there are also still many open questions which we list along with the results 
(Questions~\ref{b-non-ques}, \ref{Matet-Shelah-ques}, \ref{succ-b-d-ques}, \ref{add-cof-ques}, and~\ref{cof-ques} in
this subsection, but see also Question~\ref{add-b-ques}).

We start with a characterization of the dominating and bounding numbers, which is just a straightforward generalization of a similar result in the countable case~\cite{Bl10}. Assume $(i_\alpha: \alpha < \kappa)$ is a strictly increasing and continuous sequence of ordinals below $\kappa$ with $i_0=0$. Then $I = (  I_\alpha = [ i_\alpha, i_{\alpha + 1} ) : \alpha < \kappa)$ is an interval partition of $\kappa$. Let $\IP$ be the family of such interval partitions.

\begin{defi}
We say that an interval partition $I= ( I_\alpha =[i_\alpha , i_{\alpha + 1}) : \alpha < \kappa )$ of $\kappa$ {\em dominates} another interval partition $J=(J_\alpha = [j_\alpha, j_{\alpha + 1}): \alpha < \kappa)$ (and write $J \leq^\ast I$) if there is a $\gamma < \kappa$ such that  for all $\alpha > \gamma$  there is some  $\beta  < \kappa$ such that $ J_{\beta} \subseteq I_\alpha$. 
\end{defi}

\begin{prop}{\rm (Similar to~\cite[Theorem 2.10]{Bl10})}   \label{dom-IP}
$\D_\kappa = (\kappa^\kappa , \kappa^\kappa , \leq^* ) \equiv {(\IP , \IP , \leq^*)}$. In particular, if we let $\D'_\kappa =  (\IP , \IP , \leq^*)$, then
$\| \mathcal{D}'_\kappa \| = \mathfrak{d}_\kappa$ and $\| (\mathcal{D}_\kappa')^\bot \| = \mathfrak{b}_\kappa$.
\end{prop}

\begin{proof}
Define $\Phi_- : \IP \to \kappa^\kappa$ by $\Phi_- (I) (\gamma) = i_{\alpha + 2}$ where $\gamma \in I_\alpha = [ i_\alpha, i _{\alpha + 1})$,
for $I \in \IP$.
Also let $\Phi_+ (f) = (J_\alpha = [ j_\alpha, j_{\alpha + 1} ) : \alpha < \kappa )$ be an interval partition 
such that $\gamma \leq j_\alpha$ implies $f (\gamma)  < j_{\alpha + 1}$, for $f \in \kappa^\kappa$. 
We use these functions for both inequalities.

To show  $(\kappa^\kappa , \kappa^\kappa , \leq^* ) \preceq (\IP , \IP , \leq^*)$, assume $\Phi_- (I) \leq^* f$ for some $f \in \kappa^\kappa$
and $I \in \IP$. Let $\alpha$ be so large that $\Phi_- (I)  (j_\alpha) \leq f (j_\alpha)$. Fix $\beta$ such that $j_\alpha \in I_\beta =
[i_\beta , i_{\beta + 1})$. Then $\Phi_- (I)  (j_\alpha) = i_{\beta + 2} \leq f (j_\alpha) < j_{\alpha + 1}$. This means that the interval
$I_{\beta +1}$ is contained in $J_\alpha$, and $I \leq^* \Phi_+ (f)$ follows.

For $(\IP , \IP, \leq^*) \preceq (\kappa^\kappa , \kappa^\kappa, \leq^*)$, assume $\Phi_+ (f) \leq^* I$. Let $\gamma$ be so large that
all intervals of $I$ beyond and including the one containing $\gamma$ contain an interval of $\Phi_+ (f)$. Suppose that $\gamma$
belongs to $I_\alpha$ and $J_\beta$. Since $\gamma < j_{\beta + 1}$, this means $f (\gamma) < j_{\beta +2 }$ and the latter must
be less or equal than $i_{\alpha + 2} = \Phi_- (I) (\gamma)$ because $I_{\alpha + 1}$ contains an interval of $J$.
Thus $f \leq^* \Phi_- (I)$ follows.
\end{proof}

In order to prove some of the inequalities related to the cardinals associated to the ideal of $\ka$-meager sets we will use the approach of Blass in \cite{Bl10}. 

\begin{defi}
A {\em $\kappa$-chopped function} is a pair $(x,I)$, where $x \in 2^\kappa$ and $I=(I_\alpha = [i_\alpha , i_{\alpha+1}) : \alpha < \kappa )$ is an interval partition of $\kappa$. A real $y \in 2^\kappa$ {\em matches} a $\ka$-chopped function $(x,I)$ if $x\restriction I_\alpha = y \restriction I_\alpha$ for cofinally many  intervals $I_\alpha \in I$.
\end{defi}

In the countable case it is possible to characterize meagerness in terms of chopped reals (see~\cite[5.2]{Bl10}):
in fact, a subset $M$ of $2^\omega$ is meager if and only if there is a $\omega$-chopped function that no member of $M$ matches.

Now, we present some of the results obtained by Blass, Hyttinen, and Zhang in \cite{BHZta} that show that, in general, this is not the case when we go to the uncountable.

\begin{defi}
We say that a subset $M$ of $2^\kappa$ is {\em combinatorially meager} if there is a $\kappa$-chopped function that no member of $M$ matches. We call $\Match (x, I)$ the set of functions in $2^\kappa$ matching the $\ka$-chopped function $(x, I)$. 
\end{defi}

We include some proofs of the following results for the sake of completeness.

\begin{obs}{\rm \cite[4.6]{BHZta}}  \label{meager-chop1}
Every combinatorially meager set is meager. 
\end{obs}

\begin{proof}
Let $(x,I)$ be a $\kappa$-chopped function with interval partition $I=(I_\alpha : \alpha < \kappa )$.
Note that for every $\beta < \kappa$, the set $\{ y \in 2^\kappa : y \re I_\alpha \neq x \re I_\alpha$ for all $\alpha > \beta \}$
is nowhere dense in $2^\kappa$.
\end{proof}

\begin{prop}{\rm \cite[4.7 and 4.8]{BHZta}}  \label{meager-chop2}
$\kappa$ is strongly inaccessible if and only if every meager set is combinatorially meager. 
\end{prop}

\begin{proof}
We only show the direction from left to right since we won't use the converse.

Assume $\kappa$ is strongly inaccessible and let $A$ be meager, that is, $A = \bigcup_{\alpha < \kappa} A_\alpha$
where the $A_\alpha$ form an increasing sequence of nowhere dense sets. Since $2^\alpha < \kappa$ for all $\alpha < \kappa$, we can recursively
construct a continuous increasing sequence $(i_\alpha : \alpha < \kappa)$ of ordinals less than $\kappa$ and
a sequence of partial functions $(\sigma_\alpha \in 2^{[i_\alpha , i_{\alpha + 1} ) } : \alpha < \kappa )$ such that
for all $\tau \in 2^{i_\alpha}$, $[\tau \ha \sigma_\alpha] \cap A_\alpha = \emptyset$. If $x$ is the concatenation of the
$\sigma_\alpha$, that is, $x \re [i_\alpha, i_{\alpha + 1} ) = \sigma_\alpha$, and $I$ is the interval partition given
by the $i_\alpha$, then no member of $A$ matches $(x,I)$, and $A$ is combinatorially meager.
\end{proof}

Let $\CR$ denote the set of $\kappa$-chopped functions. 
In view of these results, we can identify $\Cov (\M_\kappa) = (2^\kappa , \M_\kappa, \in)$ with $(2^\kappa, \CR,$ does not match$)$
for strongly inaccessible $\kappa$ -- they are equivalent in the sense of Definition~\ref{morph-def} --
and we will use this for example in Corollary~\ref{non-cov-cor} below.

\begin{defi}\label{nmcv}
Let $f$ and $g$ be two functions from $\kappa$ to $\kappa$. We say that $f$ and $g$ are {\em eventually different}  (and write $f \neq^* g$) if $\lvert\{\alpha < \kappa : f(\alpha)= g(\alpha)\}\lvert < \kappa$. Otherwise we say $f$ and $g$ are {\em cofinally matching}. We define the following two cardinal invariants:
\begin{itemize}
\item $\bb_\kappa (\neq^*)= \min \{\lvert \mathcal{F}\lvert : (\forall g \in \kappa^\kappa) (\exists f \in \mathcal{F})\neg(f \neq^* g)\}$.
\item $\dd_\kappa (\neq^*)= \min \{\lvert \mathcal{F}\lvert : (\forall g \in \kappa^\kappa) (\exists f \in \mathcal{F})(f \neq^* g)\}$.
\end{itemize}
Considering the triple $\mathcal{E}_\kappa=(\kappa^\kappa, \kappa^\kappa, \neq^*)$, we see that $\| \mathcal{E}_\kappa \| = \dd_\kappa (\neq^*)$ and $\| \mathcal{E}^\bot_\kappa \| = \bb_\kappa (\neq^*)$
\end{defi}

In the countable case the cardinal invariants defined above coincide with $\non(\mathcal{M})$ and $\cov(\mathcal{M})$, respectively. See \cite[2.4.A]{BJ95}. In the uncountable we still have the following:

\begin{obs}   \label{ed-obs}
$\mathcal{D}_\kappa \preceq \mathcal{E}_\kappa \preceq \Cov (\M_\kappa)$. In particular, $\bb_\kappa \leq \bb_\kappa (\neq^*) \leq \non (\M_\kappa)$ and $\cov (\M_\kappa) \leq \dd_\kappa (\neq^*) \leq \dd_\kappa$.
\end{obs}

\begin{proof} 
The first $\preceq$ is straightforward. For the second, consider the meager ideal on $\kappa^\kappa$ instead of $2^\kappa$ and let $\Cov (\M_\kappa) = (\kappa^\kappa, \M_\kappa , \in)$. Then $\Phi_- = \id : \kappa^\kappa \to \kappa^\kappa$ and $\Phi_+ : \kappa^\kappa \to \M_\kappa$ given by $\Phi_+ (f) = \{ g \in \kappa^\kappa : g \neq^* f \}$ witness $\mathcal{E}_\kappa \preceq \Cov (\M_\kappa)$.
For an argument with the space $2^\kappa$, see~\cite[Proposition 4.13]{BHZta}.
\end{proof}

The following result is implicit in work of both Landver~\cite{La92} and Blass-Hyttinen-Zhang~\cite{BHZta} (though detailed
proofs are missing from both articles). We include the argument for the sake of completeness. As usual, for cardinals $\lambda, \mu$, and
$\nu$, we let $\Fn (\lambda , \mu , \nu)$ denote the set of partial functions from $\lambda$ to $\mu$ with domain of
size strictly less than $\nu$. 

\begin{thm}   \label{non-cov-char}
Assume $\kappa$ is strongly inaccessible. There are functions 
\[ \Phi_- : \CR \times \IP \to (( \Fn (\kappa,2,\kappa))^{<\kappa} )^\kappa\]
and 
\[ \Phi_+ : \IP \times (( \Fn (\kappa,2,\kappa))^{<\kappa} )^\kappa \to 2^\kappa\] 
such that if $(x,I) \in \CR$, $J \in \IP$, and $y \in (( \Fn (\kappa,2,\kappa))^{<\kappa} )^\kappa$  are such that 
cofinally many $J_\alpha$ contain an interval of $I$ and $\Phi_- ((x,I) , J ) (\beta) = y (\beta)$ for cofinally many $\beta$,
then $\Phi_+ (J,y)$ matches $(x,I)$.
\end{thm}

\begin{proof} 
Assume $I$ and $J$ are such that for cofinally many $\alpha$, $J_\alpha$ contains an interval of $I$. 
Let $C = \{ \alpha_\gamma : \gamma < \kappa \}$ be the enumeration of these $\alpha$. That is, for 
$\gamma < \kappa$ there is $\delta_\gamma < \kappa$ such that $I_{\delta_\gamma} \sub J_{\alpha_\gamma}$. 
Put 
\[\Phi_- ((x,I),J) (\beta) = ( x \re I_{\delta_\gamma} : \gamma < \omega_{\beta + 1} ) \]
for $\beta < \kappa$. For other $I$ and $J$, the value of $\Phi_- ((x,I),J) (\beta)$ is arbitrary.

$\Phi_+ (J,y)$ will be defined recursively. At each stage at most one interval $J_\alpha$ of $J$ will be considered and
$\Phi_+ (J,y) \re J_\alpha$ will be defined. Suppose we are at stage $\beta < \kappa$. If $y (\beta)$ is a sequence
of length $\omega_{\beta + 1}$ of partial functions all of whose domains are included in distinct $J_\alpha$'s, 
choose such $J_\alpha$ which has not been considered yet (this is possible by $|\beta | \leq \omega_\beta < \omega_{\beta + 1}$).
Then let $\Phi_+ (J,y) \re J_\alpha$ agree with the partial function from $y(\beta)$ whose domain is contained in $J_\alpha$ on its domain.
If $y (\beta)$ is not of this form, do nothing. In the end, extend $\Phi_+ (J,y)$ to a total function in $2^\kappa$ arbitrarily.

Now assume cofinally many $J_\alpha$ contain an interval of $I$ and $\Phi_- ((x,I) , J ) (\beta) = y (\beta)$ for cofinally many $\beta$.
Fix such $\beta$. Then $y (\beta)$ is a sequence of length $\omega_{\beta + 1}$ of partial functions all of whose domains 
are included in distinct $J_\alpha$'s and thus, for some $\gamma$, $\Phi_+ (J,y) \re I_{\delta_\gamma}$ will agree with
$x \re I_{\delta_\gamma}$. For different such $\beta$ we must get distinct $\gamma$, and therefore $\Phi_+ (J,y)$ matches $(x,I)$.
\end{proof}

\begin{cor}   \label{non-cov-cor}
Assume $\kappa$ is strongly inaccessible.
\begin{enumerate}
\item {\rm (Blass, Hyttinen, Zhang~\cite[4.12 and 4.13]{BHZta})} $\non (\M_\kappa) = \bb_\kappa (\neq^*)$.
\item {\rm (Landver~\cite{La92})} $\cov (\M_\kappa) = \dd_\kappa (\neq^*)$.
\end{enumerate}
\end{cor}

\begin{proof}
1. In view of Observation~\ref{ed-obs}, 
it suffices to prove $\non (\M_\kappa) \leq \bb_\kappa (\neq^*)$. Let $\Y \subseteq (( \Fn (\kappa,2,\kappa))^{<\kappa} )^\kappa$
be a family of functions of size $\bb_\kappa (\neq^*)$ which is cofinally matching. Also let $\J$ be an unbounded family of
partitions of size $\bb_\kappa \leq \bb_\kappa (\neq^*)$. We claim $\{ \Phi_+ (J,y) : J \in \J$ and $y \in \Y \}$ is nonmeager. 
Indeed, if $(x,I) \in \CR$ and $J \in \J$ is unbounded over the partition given by taking unions of pairs of intervals of $I$,
then cofinally many $J_\alpha$ contain an interval of $I$. If additionally $y \in \Y$ is such that $\Phi_- ((x,I) , J ) (\beta) = y (\beta)$ 
for cofinally many $\beta$, then $\Phi_+ (J,y)$ matches $(x,I)$ and therefore does not belong to the meager set given by $(x,I)$.

2. It suffices to prove $\dd_\kappa (\neq^*) \leq \cov (\M_\kappa)$. Let $\X \subseteq
\CR$ of size $< \dd_\kappa (\neq^*) \leq \dd_\kappa$. First choose $J \in \IP$ such that cofinally many $J_\alpha$ contain
an interval of $I$, for each $I$ such that $(x,I) \in \X$ for some $x \in 2^\kappa$. Next choose $y \in (( \Fn (\kappa,2,\kappa))^{<\kappa} )^\kappa$
such that for all $(x,I) \in \X$, $\Phi_- ((x,I) , J ) (\beta) = y (\beta)$ for cofinally many $\beta$. Then $\Phi_+ (J,y)$ matches $(x,I)$
for all $(x,I) \in \X$, and therefore does not belong to any of the meager sets given by such $(x,I)$.
\end{proof}

Shelah~\cite{Shta} proved the consistency of $\cov (\M_\kappa) < \dd_\kappa$ for supercompact $\kappa$. We do not know
whether the dual inequality is consistently strict.

\begin{ques}   \label{b-non-ques}
Assume $\kappa$ is strongly inaccessible (or even supercompact). Is $\bb_\kappa < \non (\M_\kappa)$ consistent?
\end{ques}

For successor cardinals, the situation is rather different.

\begin{thm} \label{HMS-thm} Assume $\kappa$ is a successor cardinal.
\begin{enumerate}
\item {\rm (Hyttinen~\cite{Hy06})} $\bb_\kappa (\neq^*) = \bb_\kappa$.
\item {\rm (Matet, Shelah~\cite[Theorem 4.6]{MSta})} If $2^{< \kappa} = \kappa$, then $\dd_\kappa (\neq^*) = \dd_\kappa$.
\end{enumerate}
\end{thm}

\begin{ques}   \label{Matet-Shelah-ques}
{\rm (Matet, Shelah~\cite[Section 4]{MSta}, see also~\cite[Question 3.8, part 1]{KLLSta})}
Is it consistent that $\kappa$ is a successor cardinal and $\dd_\kappa (\neq^*) < \dd_\kappa$?
\end{ques}

\begin{obs}   \label{nwd-kappa}
\begin{romanenumerate}
\item For $\sigma \in 2^{<\kappa}$, $A_\sigma = \{ x \in 2^\kappa : \forall \alpha < \kappa \; ( x \re [\alpha, \alpha + |\sigma| ) \neq \sigma) \}$
   is nowhere dense.
\item {\rm (Landver~\cite[1.3]{La92})} $2^{<\kappa} > \kappa$ implies $\add (\M_\kappa) = \cov(\M_\kappa) = \kappa^+$.
\item {\rm (Blass, Hyttinen, Zhang~\cite[4.15]{BHZta})} $\non (\M_\kappa) \geq 2^{<\kappa}$.
\end{romanenumerate}
\end{obs}

\begin{proof}
(i) is immediate. For (ii), let $\lambda < \kappa$ be such that $2^\lambda > \kappa$. Then $2^\kappa = \bigcup \{ A_\sigma :
\sigma \in \Sigma \}$ for any $\Sigma \sub 2^\lambda$ with $|\Sigma| \geq \kappa^+$. 
For (iii), fix $X \sub 2^\kappa$ with $|X| < 2^{<\kappa}$. Let $\lambda < \kappa$ be such that
$|X| < 2^\lambda$. Then $X \sub A_\sigma$ for some $\sigma \in 2^\lambda$.
\end{proof}

As a consequence one obtains that $\dd_\kappa  < \non (\M_\kappa)$ and $\cov (\M_\kappa) <
 \bb_\kappa$ are both consistent with $2^{<\kappa} > \kappa$: for the former, add $\kappa^{++}$ Cohen reals over a model
of GCH. For the latter, first force $2^\kappa = \kappa^{++} = \bb_\kappa$ by adding $\kappa$-Hechler generics 
in a $\kappa^{++}$-length iteration with $<\kappa$-supports over a model of $2^{<\kappa} = \kappa$ and $2^\kappa = \kappa^+$, and then add
$\kappa^+$ Cohen reals. In fact, if $\kappa$ is successor, then by Theorem~\ref{HMS-thm} one also has 
$\bb_\kappa (\neq^*) < \non (\M_\kappa)$ in the first model (see also~\cite[4.16]{BHZta}). We shall see that this is still
true for (weakly) inaccessible $\kappa$ and that one then also gets $\cov (\M_\kappa) < \dd_\kappa (\neq^*)$ in the
second model, see Theorem~\ref{layers}.

For successor $\kappa$ with $2^{<\kappa} = \kappa$ we do not know whether any of the two sets of three cardinals
can be separated.

\begin{ques} {\rm (see also \cite[Section 4]{MSta})}   \label{succ-b-d-ques}
Assume $\kappa = 2^{< \kappa}$ is a successor cardinal. Is $\bb_\kappa = \non (\M_\kappa)$ and $\dd_\kappa = \cov (\M_\kappa)$?
\end{ques}

\begin{prop}
$\Match(x,I) \subseteq \Match(y,J)$ if and only if for all but $<\kappa$-many intervals $I_\alpha \in I$ there exists an interval $J_\beta \in J$ such that $J_\beta \subseteq I_\alpha$ and $x \restriction J_\beta = y\restriction J_\beta$. 
\end{prop}

\begin{proof}
For the implication from left to right we argue by contradiction. Suppose there are $\kappa$ many intervals $(I_{\alpha_\gamma} : \gamma < \kappa)$ in $I$ such that for all $J_\beta \in J$ with $J_\beta \subseteq I_{\alpha_\gamma}$ for some $\gamma$, we have $x \restriction J_\beta \neq y \restriction J_\beta $. Then $x' \in \Match(x, I) \setminus \Match(y, J)$, where $x' : \kappa \to 2$ is defined by:
\begin{equation*}
 x'(\alpha) = \left\{
\begin{array}{c l}
 x(\alpha) & \alpha \in I_{\alpha_\gamma} \text{ for some  } \gamma \\
 1- y(\alpha) & \text{otherwise}. \\
\end{array}
\right.
\end{equation*}
For the other direction let $z \in \Match(x, I)$, then there are $\kappa$ many intervals $(I_{\alpha_\gamma}: \gamma < \kappa)$ in $I$ such that $z \restriction I_{\alpha_\gamma} = x \restriction I_{\alpha_\gamma}$ for $\gamma < \kappa$. Without loss of generality we can suppose that for all $\gamma$ there exists $J_{\beta_\gamma} \in J$ satisfying $J_{\beta_\gamma} \subseteq I_{\alpha_\gamma}$, and so $z \restriction J_{\beta_\gamma} = x \restriction J_{\beta_\gamma}= y \restriction J_{\beta_\gamma}$, thus $z \in \Match(y , J)$.
\end{proof}

\begin{defi}
Let $(x, I)$ and $(y, J)$ be two $\ka$-chopped functions. We say $(x, I)$ {\em engulfs} $(y, J)$ if $\Match(x,I) \subseteq \Match(y,J)$.
\end{defi}

Since complements of sets of the form $\Match (x,I)$ are meager and every meager set is contained in such a complement
(see Observation~\ref{meager-chop1} and Proposition~\ref{meager-chop2}) for strongly inaccessible $\kappa$, 
we see that $\Cof(\mathcal{M}_\kappa)= {(\mathcal{M}_\kappa,
\mathcal{M}_\kappa, \subseteq)} \equiv \Cof'(\mathcal{M}_\kappa)= (\CR, \CR, $ is engulfed by$)$. In particular, $\| \Cof(\mathcal{M}_\kappa) \| 
= \| \Cof'(\mathcal{M}_\kappa)\|$ $= \cof(\mathcal{M}_\kappa) $ and $\| \Cof(\mathcal{M}_\kappa)^\bot \| 
 =\| \Cof'(\mathcal{M}_\kappa)^{\bot}\|  = \add(\mathcal{M}_\kappa)$.

\begin{cor}
Assume $\kappa$ is strongly inaccessible. Then $\Cof (\M_\kappa) \preceq \D_\kappa$.
\end{cor}

\begin{proof}
By the previous proposition, if $(x, I)$ engulfs $(y, J)$ then $I$ dominates $J$. 
Hence $\Cof' (\M_\kappa) \preceq \D'_\kappa$.
By Proposition~\ref{dom-IP} and the previous comment $\Cof (\M_\kappa) \preceq \D_\kappa$ follows.
\end{proof}

\begin{cor}
Assume $\kappa$ is strongly inaccessible.
Then $\add (\M_\kappa) \leq \bb_\kappa$ and $\dd_\kappa \leq \cof (\M_\kappa)$.
\end{cor}

By Observation~\ref{nwd-kappa}, $\add (\M_\kappa) \leq \bb_\kappa$ also holds when $2^{<\kappa} > \kappa$.

\begin{ques}  \label{add-cof-ques}
\begin{enumerate}
\item Assume $\kappa$ is a successor cardinal with $2^{< \kappa} = \kappa$. Does $\add (\M_\kappa) \leq \bb_\kappa$?
\item Assume $\kappa$ is successor or $2^{< \kappa} > \kappa$. Does $\dd_\kappa \leq \cof (\M_\kappa)$?
\end{enumerate}
\end{ques}

Equip $2^\kappa$ with addition $+$ modulo $2$. For $A \sub 2^\kappa$ and $y \in 2^\kappa$, let
$A + y = \{ x + y : x \in A \}$. If also $B \sub 2^\kappa$, put $A + B = \{  x + y : x \in A$ and $y \in B \}$.
Finally identify $\sigma \in 2^{< \kappa}$ with the function in $2^\kappa$ which agrees with
$\sigma$ on its domain and takes value $0$ elsewhere.

\begin{prop}   \label{Truss-kappa}
Assume $2^{<\kappa} = \kappa$.
There are functions $\Phi_+ : 2^\kappa \times \kappa^\kappa \to \M_\kappa $ and
$\Phi_- : 2^\kappa \times \NWD_\kappa \to \kappa^\kappa $ such that
$x \notin B + 2^{< \kappa}$ and $f \geq^* \Phi_- (x,B)$ imply $B \sub \Phi_+ (x,f)$.
\end{prop}

\begin{proof}
Let $\{ \sigma_\beta : \beta < \kappa \}$ list $2^{<\kappa}$. Put
\[ \Phi_+ (x,f) = \bigcup_{\alpha < \kappa} \bigcap_{\beta \geq \alpha} 2^\kappa \sem [(\sigma_\beta + x) \re f(\beta) ].\]
This is clearly a meager set.
For $x \notin B + 2^{< \kappa}$ let 
$\Phi_- (x,B) (\alpha)$ be such that $B \cap [ (\sigma_\alpha + x) \re \Phi_- (x,B) (\alpha) ] = \emptyset$.
If $x \in B + 2^{< \kappa}$, the definition of $\Phi_- (x,B)$ is irrelevant.

Now assume $x \notin B + 2^{< \kappa}$ and $f \geq^* \Phi_- (x,B)$. 
Let $y \in B$. Then $y \notin [ (\sigma_\alpha + x) \re \Phi_- (x,B) (\alpha ) ]$ for all $\alpha$.
Since $f \geq^* \Phi_- (x,B)$, there is $\alpha$ such that $y \in 2^\kappa \sem
[(\sigma_\beta + x) \re f (\beta)]$ for all $\beta \geq \alpha$. Thus $y \in \Phi_+(x,f)$
as required.
\end{proof}

\begin{cor}   \label{Truss-cor}
\begin{enumerate} 
\item $\add (\M_\kappa) \geq \min \{ \bb_\kappa , \cov (\M_\kappa) \}$.
\item Assume $2^{<\kappa} = \kappa$. Then $\cof (\M_\kappa) \leq \max \{ \dd_\kappa , \non (\M_\kappa) \}$.
\end{enumerate}
\end{cor}

\begin{proof} 
1. If $2^{<\kappa} > \kappa$ this is immediate by Observation~\ref{nwd-kappa}. If $2^{<\kappa} =\kappa$ 
use Proposition~\ref{Truss-kappa}: If $\B \sub \NWD_\kappa$ with $|\B| < \min \{ \bb_\kappa , \cov (\M_\kappa) \}$, 
find $x \in 2^\kappa \setminus \bigcup \B + 2^{<\kappa}$ and then $f \in \kappa^\kappa$ with $f \geq^* \Phi_- (x,B)$
for all $B \in \B$. Then $\bigcup \B \sub \Phi_+ (x,f)$.

2. Note that if $\F  \sub \kappa^\kappa$ is dominating and $X \sub 2^\kappa$ is non-meager, then, by Proposition~\ref{Truss-kappa},
$\{ \Phi_+ (x,f) : f \in \F$ and $x \in X \}$ is a cofinal family.
\end{proof}

\begin{ques}   \label{cof-ques}
Does $\cof (\M_\kappa) \leq \max \{ \dd_\kappa , \non (\M_\kappa) \}$ also hold when $2^{< \kappa} > \kappa$?
\end{ques}

A more fundamental question might be whether $\non (\M_\kappa)$ and $\cof (\M_\kappa)$ can be
different for $\kappa$ with $2^{< \kappa} > \kappa$.

\begin{cor}
Assume $\kappa$ is strongly inaccessible.
Then $\add (\M_\kappa) = \min \{ \bb_\kappa , \cov (\M_\kappa) \}$ and 
$\cof (\M_\kappa) = \max \{ \dd_\kappa , \non (\M_\kappa) \}$.
\end{cor}


\subsection{Slaloms}

In this subsection, $\kappa$ is always a -- possibly weakly -- inaccessible cardinal.

The classical Cicho\'n diagram also contains cardinal invariants related to measure. While
there are various attempts to generalize the ideal of null sets when $\kappa$ is an inaccessible cardinal (see e.g. \cite{ShIn} and \cite{FrLa}),
we shall not pursue this but rather consider generalizations of cardinal invariants which are combinatorial characterizations
of the measure invariants in the countable case, similar to $\bb_\kappa (\neq^*)$ and $\dd_\kappa (\neq^*)$ for the category invariants.

\begin{defi}
A function $\varphi$ with $\dom (\varphi) = \kappa$ and $\varphi (\alpha) \in [\kappa]^{|\alpha|}$ for $\alpha < \kappa$ is called
a {\em slalom}. By $\Loc_\kappa$ we denote the collection of all slaloms. More generally, if $h \in \kappa^\kappa$ is a function
with $\lim_{\alpha \to \kappa} h(\alpha) = \kappa$, and $\varphi (\alpha) \in [\kappa]^{| h (\alpha) |}$ for all $\alpha < \kappa$,
then $\varphi$ is an {\em $h$-slalom}. $\Loc_h$ is the set of $h$-slaloms. So $\Loc_\kappa = \Loc_\id$. A slalom $\varphi$
{\em localizes} a function $f \in \kappa^\kappa$ ($f \in^* \varphi$ in symbols) if $| \{ \alpha < \kappa : f(\alpha) \notin \varphi (\alpha) \} | < \kappa$. 
We define:
\begin{itemize}
\item $\bb_h (\in^*)= \min \{\lvert \mathcal{F}\lvert : \F \subseteq \kappa^\kappa$ and 
   $(\forall \varphi \in \Loc_h) (\exists f \in \mathcal{F})\neg(f \in^* \varphi)\}$.
\item $\dd_h (\in^*)= \min \{\lvert \Phi \lvert : \Phi \subseteq \Loc_h$ and $(\forall f \in \kappa^\kappa) (\exists \varphi \in \Phi)(f \in^* \varphi)\}$.
\end{itemize}
Considering the triple $\LOC_h = (\kappa^\kappa, \Loc_h , \in^*)$, we see that $|| \LOC_h || = \dd_h (\in^*)$ and $|| \LOC_h^\bot || = \bb_h (\in^*)$.
\end{defi}

In the countable case these cardinals coincide with $\add (\N)$ and $\cof (\N)$, respectively, for arbitrary functions $h \in \omega^\omega$
such that $\lim h(n) = \infty$~\cite[2.3]{BJ95}. In particular, all $\bb_h (\in^*)$ are equal, and so are all $\dd_h (\in^*)$. We shall prove
below this is false for strongly inaccessible $\kappa$, see Theorem~\ref{Sacks-model}.

\begin{defi}
Again let $h \in \kappa^\kappa$ with $\lim_{\alpha \to \kappa} h(\alpha) = \kappa$. A function $\varphi$ is a {\em partial
$h$-slalom} if $\dom (\varphi) \subseteq \kappa$, $|\dom (\varphi) | = \kappa$ and $\varphi (\alpha) \in [\kappa]^{|h(\alpha)|}$ for $\alpha 
\in \dom (\varphi)$. $\pLoc_h$ is the set of partial $h$-slaloms. For $h = \id$, write $\pLoc_\kappa = \pLoc_\id$. 
$\varphi$ {\em localizes} $f \in \kappa^\kappa$ ($f \in^* \varphi$ in symbols) if $| \{ \alpha \in \dom (\varphi) : 
f(\alpha) \notin \varphi (\alpha) \} | < \kappa$. The corresponding cardinals are:
\begin{itemize}
\item $\bb_h (\p\in^*)= \min \{\lvert \mathcal{F}\lvert : \F \subseteq \kappa^\kappa$ and 
   $(\forall \varphi \in \pLoc_h) (\exists f \in \mathcal{F})\neg(f \in^* \varphi)\}$.
\item $\dd_h (\p\in^*)= \min \{\lvert \Phi \lvert : \Phi \subseteq \pLoc_h$ and $(\forall f \in \kappa^\kappa) (\exists \varphi \in \Phi)(f \in^* \varphi)\}$.
\end{itemize}
If $\pLOC_h = (\kappa^\kappa, \pLoc_h , \in^*)$, we see that $|| \pLOC_h || = \dd_h (\p\in^*)$ and $|| \pLOC_h^\bot || = \bb_h (\p\in^*)$.
\end{defi}

Formally, we could also have defined $\LOC_\kappa$ and $\pLOC_\kappa$ for successors. However, it is easy to see
that in this case $\LOC_\kappa \equiv \pLOC_\kappa \equiv \D_\kappa$ so that the resulting cardinals are equal to
$\bb_\kappa$ and $\dd_\kappa$, respectively, and thus not interesting. For (weakly) inaccessible $\kappa$ we still have:

\begin{obs}
$\LOC_h \preceq \pLOC_h \preceq \D_\kappa$. So $\bb_h (\in^*) \leq \bb_h (\p\in^*) \leq \bb_\kappa$ and 
$\dd_\kappa \leq \dd_h (\p\in^*) \leq \dd_h (\in^*)$.
\end{obs}

We shall prove that the inequalities regarding the two different kinds of localization cardinals 
can be consistently strict like in the countable case, see Theorem~\ref{total-partial-thm}.

Again, in the countable case, the partial localization cardinals do not depend on the function $h$~\cite{evpr2}. This is still true for 
(weakly) inaccessible $\kappa$.

\begin{obs} 
Let $g$ and $h$ with $\lim_{\alpha \to \kappa} g(\alpha) = \lim_{\alpha \to \kappa} h(\alpha) = \kappa$. Then
$\pLOC_g \equiv \pLOC_h$.
\end{obs}

\begin{proof}
It suffices to show that $\pLOC_g \preceq \pLOC_h$. Choose strictly increasing $\alpha_\gamma$ for $\gamma < \kappa$ with 
$h(\alpha_\gamma) \geq g(\gamma)$. Given $f \in \kappa^\kappa$, let $\Phi_- (f) (\gamma) = f (\alpha_\gamma)$ for all $\gamma
< \kappa$. If $\varphi \in \pLoc_g$, let $\dom ( \Phi_+ (\varphi)) = \{ \alpha _\gamma : \gamma \in \dom (\varphi) \}$
and $\Phi_+ (\varphi) (\alpha_\gamma) = \varphi (\gamma) \in [\kappa]^{| g (\gamma)|} \subseteq [\kappa]^{|h (\alpha_\gamma)|}$
for $\gamma \in \dom (\varphi)$. If $\Phi_- (f) \in^* \varphi$, then $f \in^* \Phi_+ (\varphi)$ because $f (\alpha_\gamma) =
\Phi_- (f) (\gamma) \in \varphi (\gamma) = \Phi_+ (\varphi) (\alpha_\gamma)$ holds for all large enough $\gamma \in \dom (\varphi)$.
\end{proof}

\begin{cor} 
Assume $\kappa$ is (weakly) inaccessible. For any function $h \in \kappa^\kappa$ with $\lim_{\alpha \to \kappa} h(\alpha) = \kappa$,
$\bb_h (\p \in^*) = \bb_\kappa (\p \in^*)$ and $\dd_h (\p \in^*) = \dd_\kappa (\p \in^*)$.
\end{cor}

For the remainder of this subsection, we assume that $\kappa$ is a strongly inaccessible cardinal.

The classical Bartoszy\'nski-Raisonnier-Stern Theorem~\cite[Theorem 2.3.1]{BJ95} says that $\Cof (\N) \preceq \Cof (\M)$. 
This is proved by showing $\Cof (\N) \equiv \LOC_\omega$ and $\LOC_\omega \preceq \Cof (\M)$.
In fact, an analysis of the proof (see e.g.~\cite[2.5]{evpr2}) shows that $\pLOC_\omega \preceq \Cof (\M)$.
This is the version of the Bartoszy\'nski-Raisonnier-Stern Theorem we shall generalize to inaccessible $\kappa$.

\begin{mainlem}
Assume $\kappa$ is strongly inaccessible. 
Let $X \sub 2^\kappa$ be a non-empty open set and let $\lambda < \kappa$. 
Then there is a family $\Y$ of open subsets of $X$ such that
\begin{romanenumerate}
\item $| \Y | \leq \kappa$,
\item every dense open subset of $2^\kappa$ contains a member of $\Y$,
\item $\bigcap \Y' \neq\emptyset$ for any $\Y' \sub \Y$ with $|\Y' | \leq \lambda$.
\end{romanenumerate}
\end{mainlem}

\begin{proof} Let $\{ X_\alpha : \alpha < \kappa \}$ list all $<\kappa$-unions of basic open sets
(i.e., sets of the form $[\sigma]$ for $\sigma \in 2^{<\kappa}$) of the relative topology of $X$.
Say $X_\alpha = \bigcup \{ [\sigma] : \sigma \in \Sigma_\alpha \}$ where $\Sigma_\alpha \sub 2^{<\kappa}$
with $|\Sigma_\alpha| < \kappa$. For simplicity assume $X = 2^\kappa$.

For $\beta < \kappa$ put
\[ A_\beta = \{ \alpha :  \forall \sigma \in 2^\beta \; \exists \tau \in 2^{<\kappa} \; (\tau \supseteq \sigma \land
\tau \in \Sigma_\alpha ) \}. \]
Since $|2^\beta| <\kappa$ (by inaccessibility of $\kappa$), $A_\beta$ is non-empty.
Also the $A_\beta$ form a decreasing chain of subsets of $\kappa$. 

Next let
\[ \Y = \{ \bigcup_{\zeta < \lambda^+} X_{\alpha_\zeta} : \alpha _0 \in \kappa \mbox{ and } \alpha_\zeta \in A_{\beta_\zeta}
\mbox{ for } \zeta > 0 \mbox{ where } \beta_\zeta = \bigcup_{\xi < \zeta} \bigcup_{\sigma \in \Sigma_{\alpha_\xi}} \dom (\sigma) \}. \]
Since $\lambda^+ < \kappa$ we see that $\Y$ has size at most $\kappa$. We need to check the other two
properties.

(ii) Let $D \sub 2^\kappa$ be open dense. First note that $ \{ \alpha \in A_\beta : X_\alpha \sub D \} \neq \emptyset$
for all $\beta < \kappa$. To see this fix $\beta$. For all $\sigma \in 2^\beta$ find $\tau_\sigma \supseteq \sigma$
with $[\tau_\sigma] \sub D$ by the density of $D$. Next let $\alpha$ be such that $\{ \tau_\sigma : \sigma \in
2^\beta \} = \Sigma_\alpha$. Then $X_\alpha \sub D$ and $\alpha \in A_\beta$ follow.

This means we can recursively construct $\alpha_\zeta$, $\zeta < \lambda^+$, such that $\alpha_\zeta \in A_{\beta_\zeta}$
and $X_{\alpha_\zeta} \sub D$. Clearly $\bigcup_{\zeta < \lambda^+} X_{\alpha_\zeta} \in \Y$
is a subset of $D$.

(iii) Assume $Y_\delta \in \Y$, $\delta < \lambda$, are given.
Say $Y_\delta = \bigcup_{\zeta < \lambda^+} X_{\alpha(\delta,\zeta)}$ where $\alpha(\delta,0) \in \kappa$
and $\alpha (\delta , \zeta) \in A_{\beta (\delta , \zeta)}$ for $\zeta > 0$ with $\beta(\delta,\zeta) = \bigcup_{\xi < \zeta}
\bigcup_{\sigma \in \Sigma_{\alpha (\delta,\xi)}} \dom (\sigma)$. 

Recursively define a partial injective function $\eta : \lambda^+ \to \lambda$ by:
\[\begin{array}{rcl}
\eta (0) & = & \min \{ \delta : \forall \epsilon < \lambda \;( \beta(\delta,1) \leq \beta(\epsilon,1) )\} \\
\eta (\zeta) & = & \min \{ \delta \notin \{ \eta (\xi) : \xi < \zeta \} : \forall \epsilon \notin \{ \eta (\xi) :
   \xi  < \zeta \} \; (\beta(\delta,\zeta+1) \leq \beta (\epsilon,\zeta+1) ) \} \\
\end{array} \]
for $\zeta > 0$ as long as $\{ \eta (\xi) : \xi < \zeta \} \subsetneq \lambda$. Let $\lambda_0$ be minimal
such that $\eta (\lambda_0)$ is undefined. Then $\lambda \leq \lambda_0 < \lambda^+$ and $\eta : \lambda_0 \to \lambda$
is a bijection. We claim that the intersection $\bigcap_{\zeta<\lambda_0} X_{\alpha (\eta(\zeta),\zeta)}$ is non-empty.
This finishes the proof because $\bigcap_{\zeta<\lambda_0} X_{\alpha (\eta(\zeta),\zeta)} \sub
\bigcap_{\zeta < \lambda_0} Y_{\eta(\zeta)} = \bigcap_{\delta<\lambda} Y_\delta$.

To prove the claim, recursively construct $( \sigma_\zeta \in 2^{<\kappa}: \zeta < \lambda_0 )$ such that
\begin{itemize}
\item $\sigma_0 = \la\ra$, $\sigma_\xi \sub \sigma_\zeta$ for $\xi < \zeta$,
\item $\sigma_\zeta = \bigcup_{\xi < \zeta } \sigma_\xi$ for limit $\zeta$,
\item $\sigma_{\zeta + 1} \in \Sigma_{\alpha(\eta (\zeta) , \zeta )}$,
\item $\dom (\sigma_\zeta) \sub \bigcup_{\xi < \zeta} \beta (\eta (\xi), \xi + 1)$.
\end{itemize}
Once this is done we let $\sigma = \bigcup_{\zeta< \lambda_0} \sigma_\zeta$. Then clearly $[\sigma] \sub
\bigcap_{\zeta<\lambda_0} X_{\alpha (\eta (\zeta) , \zeta)}$, and the claim is established.

Before starting the recursion, note that for $\xi < \zeta$ we have
\[ \beta(\eta (\xi) , \xi +1) \leq \beta (\eta (\zeta), \xi +1) \leq \beta (\eta(\zeta),\zeta) \leq \beta(\eta(\zeta),\zeta+1) \]
where the first inequality holds by the definition of $\eta$ and the other two, by the monotonicity of the
function $\beta(\eta(\zeta) , \cdot)$. 

For $\zeta = 0$ put $\sigma_0 = \la\ra$. There is nothing to verify. For $\zeta = 1$, let $\sigma_1 \in \Sigma_{\alpha(\eta
(0),0)}$. So $\dom (\sigma_1) \sub \beta(\eta (0),1)$ by definition of $\Y$. 

In general assume $\sigma_\zeta$ has been defined as required. In particular $\dom (\sigma_\zeta) \sub \bigcup_{\xi<\zeta}
\beta (\eta(\xi),\xi + 1) \sub \beta (\eta (\zeta) , \zeta)$. By definition of $\Y$ we know $\alpha (\eta(\zeta) , \zeta) \in
A_{\beta (\eta (\zeta),\zeta)}$. Hence, by definition of $A_{\beta (\eta (\zeta),\zeta)}$, we can find 
$\sigma_{\zeta +1} \in \Sigma_{\alpha (\eta (\zeta) , \zeta ) }$ with $\sigma_{\zeta + 1} \supseteq \sigma_\zeta$.
So $\dom (\sigma_{\zeta +1} ) \sub \beta (\eta(\zeta),\zeta +1)$ by definition of $\Y$.

For limit $\zeta$, let $\sigma_\zeta = \bigcup_{\xi < \zeta} \sigma_\xi$. Then $\dom (\sigma_\zeta) =
\bigcup_{\xi < \zeta} \dom (\sigma_\xi) \sub \bigcup_{\xi < \zeta} \beta(\eta(\xi),$ $\xi + 1)$ as required.
This completes the recursive construction, and the proof of the lemma.
\end{proof}

\begin{thm}   \label{BRT-thm}
Assume $\kappa$ is strongly inaccessible. Then $\pLOC_\kappa \preceq \Cof (\M_\kappa)$. That is,
there are functions $\Phi_- : \M_\kappa \to \kappa^\kappa$ and $\Phi_+ : \pLoc_\kappa \to \M_\kappa$ 
such that $\Phi_- (A)  \in^* \varphi$ implies $A \sub \Phi_+(\varphi)$ for $A \in \M_\kappa$ and $\varphi \in \pLoc_\kappa$.
\end{thm}

\begin{proof}
Identifying $\kappa$ with $\kappa^{<\kappa}$ we can work with the space $\{ f \in (\kappa^{<\kappa})^\kappa
: f (\beta) \in \kappa^\beta$ for all $\beta\}$ instead of $\kappa^\kappa$.
Let $\lambda_\beta = | \beta |$ for $\beta < \kappa$.
Assume $\{ X_\alpha : \alpha < \kappa \}$ is a basis for the topology on $2^\kappa$.
Let $\Y_{\alpha,\beta} = \{ Y_{\alpha, \beta,\gamma } : \gamma < \kappa \}$ be a family 
as in the preceding lemma with $X = X_\alpha$ and $\lambda = \lambda_\beta$
where $\alpha < \beta$. 

Given $A \in \M_\kappa$, say $A = \bigcup_{\alpha < \kappa} A_\alpha$, where all $A_\alpha$ are
nowhere dense and $A _\alpha \sub A_\beta$ for $\alpha < \beta$, we define $\Phi_- (A) \in (\kappa^{<\kappa})^\kappa$
with $\Phi_- (A) (\beta) \in \kappa^\beta$ by stipulating that $A_\beta \cap Y_{\alpha,\beta,\Phi_- (A)(\beta)(\alpha)} = \emptyset$
for all $\alpha < \beta$. By property (ii) of the preceding lemma, we can indeed find $\Phi_- (A) (\beta) (\alpha) \in \kappa$
for each $\alpha < \beta$ (for fixed $\beta$). 

Given $\varphi \in \pLoc_\kappa$ with $\varphi (\beta) \in [ \kappa^\beta ]^{\lambda_\beta}$ for all $\beta \in \dom (\varphi)$,
put
\[ \Phi_+(\varphi) = 2^\kappa \sem \bigcap_{\delta < \kappa} \bigcup_{\beta \geq \delta, \beta \in \dom (\varphi) }
\left( \bigcup_{\alpha < \beta} \bigcap_{\sigma \in \varphi (\beta) } Y_{\alpha, \beta , \sigma (\alpha) } \right). \]
By property (iii) of the preceding lemma, intersections of the form $\bigcap_{\sigma \in \varphi (\beta) } Y_{\alpha, \beta , \sigma (\alpha) }$
are non-empty, and this implies that sets of the form $$\bigcup_{\beta \geq \delta, \beta \in \dom (\varphi) }
\left( \bigcup_{\alpha < \beta} \bigcap_{\sigma \in \varphi (\beta) } Y_{\alpha, \beta , \sigma (\alpha) } \right)$$
are open dense. Hence $\Phi_+(\varphi)$ is indeed meager.

Now assume that $\Phi_- (A)  \in^* \varphi$, that is, $\Phi_- (A) (\beta) \in \varphi (\beta)$ for almost all $\beta \in \dom (\varphi)$.
Say $\beta_0$ is such that $\Phi_- (A) (\beta) \in \varphi (\beta)$ for all $\beta \geq \beta_0$ in $\dom(\varphi)$. 
Let $x \in A$. We need to show that $x$ belongs to $\Phi_+(\varphi)$. Let $\delta < \kappa$ be
such that $x \in A_\delta$. We may assume $\delta \geq \beta_0$. Clearly $x \in A_\beta$ for all $\beta \geq \delta$. Fix $\beta
\geq\delta$ in $\dom(\varphi)$. Then $x \notin Y_{\alpha,\beta,\Phi_- (A) (\beta)(\alpha) }$ for all $\alpha < \beta$. In particular,
$x \notin \bigcup_{\alpha < \beta} \bigcap_{\sigma \in \varphi(\beta)} Y_{\alpha,\beta,\sigma(\alpha) }$. 
Since this holds for all $\beta \geq \delta$ with $\beta \in \dom (\varphi)$ we see that indeed $x \in \Phi_+(\varphi)$.
This completes the proof.
\end{proof}

\begin{cor}   \label{BRT-cor}
Assume $\kappa$ is strongly inaccessible.
Then $\bb_\kappa (\p \in^*) \leq \add (\M_\kappa)$ and $\cof (\M_\kappa) \leq \dd_\kappa (\p \in^*)$.
\end{cor}

We shall see in Proposition~\ref{layers2} that $2^{<\kappa} = \kappa$ is indeed needed for this result,
for consistently there is a weakly inaccessible cardinal $\kappa$ for which $\add (\M_\kappa) < \bb_\kappa (\p\in^*) =
\dd_\kappa (\p\in^*) < \cof (\M_\kappa)$.

\begin{ques}   \label{partial-add-question}
Assume $\kappa$ is strongly inaccessible.
Are $\bb_\kappa (\p \in^*) < \add (\M_\kappa)$ and $\cof (\M_\kappa) < \dd_\kappa (\p \in^*)$ consistent?
\end{ques}

For $\kappa = \omega$, models for the two inequalities are the Hechler and the dual Hechler models, 
respectively~\cite[Theorem A]{evpr} (see also the discussion in Subsection~\ref{Hechler}).

The cardinals we discussed in this section can be displayed in the following diagram.

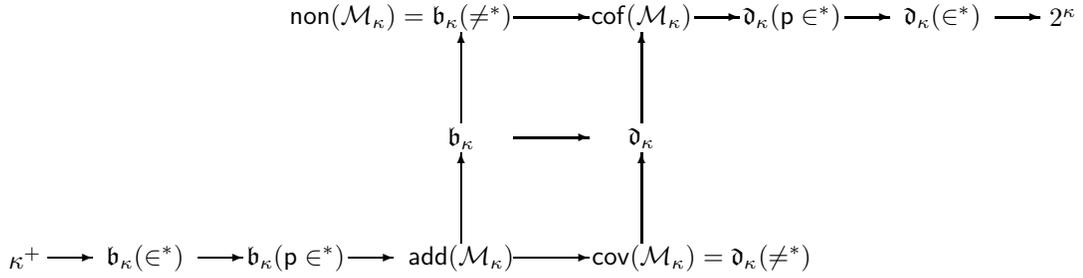
\begin{figure}[ht]
\begin{center}
\setlength{\unitlength}{0.2000mm}
\begin{picture}(800.0000,180.0000)(50,10)
\thinlines
\put(675,180){\vector(1,0){30}}
\put(575,180){\vector(1,0){30}}
\put(475,180){\vector(1,0){30}}
\put(355,180){\vector(1,0){50}}
\put(355,100){\vector(1,0){50}}
\put(245,20){\vector(1,0){30}}
\put(355,20){\vector(1,0){50}}
\put(145,20){\vector(1,0){30}}
\put(45,20){\vector(1,0){30}}
\put(440,30){\vector(0,1){60}}
\put(440,110){\vector(0,1){60}}
\put(320,110){\vector(0,1){60}}
\put(320,30){\vector(0,1){60}}
\put(610,170){\makebox(60,20){$\dd_\kappa (\in^*)$}}
\put(510,170){\makebox(60,20){$\dd_\kappa (\p\in^*)$}}
\put(410,170){\makebox(60,20){$\cof(\mathcal{M}_\kappa)$}}
\put(410,90){\makebox(60,20){$\mathfrak{d}_\kappa$}}
\put(430,10){\makebox(100,20){$\cov(\mathcal{M}_\kappa) = \dd_\kappa (\neq^*)$}}
\put(290,10){\makebox(60,20){$\add(\mathcal{M}_\kappa)$}}
\put(290,90){\makebox(60,20){$\mathfrak{b}_\kappa$}}
\put(230,170){\makebox(100,20){$\non(\mathcal{M}_\kappa) = \bb_\kappa (\neq^*)$}}
\put(190,10){\makebox(40,20){$\bb_\kappa (\p\in^*)$}}
\put(90,10){\makebox(40,20){$\bb_\kappa (\in^*)$}}
\put(10,10){\makebox(40,20){$\kappa^+$}}
\put(700,170){\makebox(40,20){$2^\kappa$}}
\end{picture}
\end{center}
\caption{Cicho\'n's diagram for strongly inaccessible $\kappa$}
\label{ODiag}
\end{figure}



\section{Iterations with support of size $<\kappa$}

In this section, we will look at iterations with support of size $< \kappa$ of $\kappa^+$-cc and $<\kappa$-closed forcing and their effect
on the cardinals introduced and studied in the previous section. Many constructions are a natural generalization
of models obtained by finite support iteration of ccc forcing for the case $\kappa = \omega$. However, unlike 
the countable case, we do not have a powerful theory of preservation theorems for iterations and therefore can
compute all cardinals only in a few cases (e.g. Propositions~\ref{Cohen-model} and~\ref{slalom-model} and 
Theorem~\ref{total-partial-thm} below), and many questions about consistency remain open.
Like in the countable case, we have {\em duality}, that is, if a long iteration of a forcing over a model of GCH increases a cardinal 
$|| \AA ||$ and makes it equal to $2^\kappa$, 
then a short iteration of the same forcing over a model of $2^{<\kappa} = \kappa$ with large $2^\kappa$
makes the dual cardinal $|| \AA^\bot ||$ of size $\kappa^+ < 2^\kappa$ (see Proposition~\ref{slalom-model} and 
Theorem~\ref{total-partial-thm} below).

We will use the following strengthening of the $\lambda^+$-cc.

\begin{defi}
Let $\kappa,\lambda$ be cardinals and $\PP$ a p.o. A set $P \sub \PP$ is called {\em $<\kappa$-centered} if any $< \kappa$ many
conditions in $P$ have a lower bound in $\PP$. So $P$ is $<\omega$-centered if it is {\em centered} in the usual sense.
$\PP$ is {\em $(\lambda , < \kappa)$-centered} if $\PP = \bigcup_{\alpha<\lambda} P_\alpha$ where all $P_\alpha$
are $< \kappa$-centered. If $\kappa = \lambda$, we say that $\PP$ is {\em $\kappa$-centered}.
Thus $\PP$ is {\em $\sigma$-centered} if it is $\omega$-centered.
\end{defi}

The following basic facts are well-known.

\begin{lem}   \label{basic}
Assume $\lambda \geq \kappa^+$ is a regular cardinal, and $( \PP_\alpha , \dot \QQ_\alpha : \alpha < \lambda)$ is an iteration
with supports of size $< \kappa$ such that all $\PP_\alpha$, $\alpha \leq \lambda$, are $\kappa^+$-cc. Also assume that
$\mu < \lambda$ and $A \sub \mu$ belongs to the $\PP_\lambda$-generic extension. Then there is $\alpha < \lambda$ such that
$A$ belongs to the $\PP_\alpha$-generic extension.
\end{lem}

\begin{lem}     \label{basic2}
Assume $\lambda \geq \kappa^+$ is a regular cardinal, and $( \PP_\alpha , \dot \QQ_\alpha : \alpha < \lambda)$ is an iteration
with supports of size $< \kappa$ such that all $\dot \QQ_\alpha$, $\alpha < \lambda$, are forced to be nontrivial and $< \kappa$-closed.
Then $\PP_\lambda$ adds a generic for $\Fn (\lambda, 2 , \kappa)$, the forcing for adding $\lambda$ many $\kappa$-Cohen
functions (see below).
\end{lem}


\subsection{Generalized Cohen forcing}

Generalized Cohen forcing is the most basic forcing notion for blowing up $2^\kappa$ and goes back Cohen's work on
the consistency of the failure of GCH. Assume $\kappa = 2^{< \kappa}$ is a regular uncountable cardinal.

\begin{defi}
Define $\CC_\kappa = \Fn (\kappa, 2 , \kappa)$, the partial functions from $\kappa$ to $2$ with domain of
size less than $\kappa$, ordered by reverse inclusion, i.e., $s \leq t$ if $s \supseteq t$ for
$s,t \in \CC_\kappa$.  More generally, for a set $A$ of ordinals, let $\CC_\kappa^A = \Fn (A \times \kappa,
2, \kappa)$ with the same ordering.
\end{defi}

$\CC_\kappa$ generically adds a {\em $\kappa$-Cohen function} $c = c_G$ given by $c = \bigcup G$ where $G$
is the $\CC_\kappa$-generic filter. Similarly, for each $\gamma \in A$, $\CC_\kappa^A$ adds  a $\kappa$-Cohen function
$c_\gamma$ given by $c_\gamma = \bigcup \{ s \re (\{ \gamma \} \times \kappa): s \in G \}$.

It is well-known and easy to see that $\CC_\kappa^A$ is $<\kappa$-closed and has the $\kappa^+$-cc
(the latter uses of course $2^{<\kappa} = \kappa$). We also note that for an ordinal $\mu$,
$\CC_\kappa^\mu$ is forcing equivalent to the $\mu$-stage iteration of $\CC_\kappa$ with
supports of size $< \kappa$. The values of the cardinal invariants in the $\kappa$-Cohen model
are easy to compute and probably known. We include the argument for the sake of completeness.

Given $f : 2^{<\kappa} \to 2^{<\kappa}$ such that $\sigma \sub f (\sigma)$ for all $\sigma \in 2^{< \kappa}$,
let $A_f = \{ x \in 2^\kappa : f (\sigma) \not\sub x$ for all $\sigma \in 2^{< \kappa } \}$. Then $A_f$ is closed
nowhere dense and for every nowhere dense set $B$ there is $f$ such that $B \sub A_f$. Thus it suffices to
consider nowhere dense sets of the form $A_f$.

\begin{prop}   \label{Cohen-model}
Let $\kappa = 2^{<\kappa}$ be regular uncountable. Also let  $\lambda > \kappa^+$ with $\lambda^\kappa = \lambda$.
Then, in the $\CC_\kappa^\lambda$-generic extension, $\non (\M_\kappa) = \kappa^+ < \cov (\M_\kappa) =
2^\kappa = \lambda$ holds.
\end{prop}

\begin{proof}
$2^\kappa = \lambda$ is well-known.

Let $\dot f: 2^{< \kappa} \to 2^{<\kappa}$ be a $\CC_\kappa^\lambda$-name for a function with 
$\sigma \sub \dot f (\sigma)$ for all $\sigma \in 2^{< \kappa}$. By the $\kappa^+$-cc there is
$B_{\dot f} \sub \lambda$ of size at most $\kappa$ such that $\dot f$ is already added by
$\CC_\kappa^{B_{\dot f}}$ (see also Lemma~\ref{basic}). Furthermore, if $\beta \notin B_{\dot f}$, then $\dot c_\beta$
is forced not to belong to $A_{\dot f}$. 

(To see this, let $s \in \CC_\kappa^\lambda$. Take
$\sigma \in 2^{< \kappa}$ such that $s (\beta , \cdot) \sub \sigma$. Strengthen $ s \re B_{\dot f}
\times \kappa \in \CC_\kappa^{B_{\dot f}}$ to $t_0 \in \CC_\kappa^{B_{\dot f}}$ such that
$t_0$ decides $\dot f (\sigma)$, say $t_0 \forces \dot f (\sigma) = \tau$. Extend $s (\beta , \cdot)$ to 
$t_1 = \tau$. Define $t \leq s$ by $t \re B_{\dot f} \times \kappa  = t_0$, $t \re \{ \beta \} \times \kappa = t_1$,
and $t \re ( \lambda \sem ( B_{\dot f} \cup \{ \beta \} ))\times \kappa = s \re ( \lambda \sem ( B_{\dot f} \cup \{ \beta \} ))\times \kappa$.
Clearly $t$ forces $\dot c_\beta \notin A_{\dot f}$.)

Now, if $\mu < \lambda$, and $\dot f_\gamma$, $\gamma < \mu$, are such names,
then, for $\beta \in \lambda \sem \bigcup_{\gamma < \mu} B_{\dot f_\gamma}$, 
$\dot c_\beta$ will witness that $\bigcup A_{\dot f_\gamma} \neq 2^\kappa$,
and $\cov (\M_\kappa) \geq \lambda$ follows.

To see $\non (\M_\kappa) \leq \kappa^+$, we show that $\dot C = \{ \dot c_\beta : \beta < \kappa^+ \}$ is a non-meager
set in the generic extension. Indeed, if $\dot f_\gamma$, $\gamma < \kappa$, are names as before
and $\beta \in \kappa^+ \sem \bigcup_{\gamma < \kappa} B_{\dot f_\gamma}$, then
$\dot c_\beta$ will witness that $\dot C$ is not contained the union of the $A_{\dot f_\gamma}$.
\end{proof}

Notice that by the results in Section 3, we know all cardinals in this model. Namely,
$\bb_h (\in^*) = \bb_\kappa (\p \in^*) = \add (\M_\kappa) = \bb_\kappa = \bb_\kappa (\neq^*) = \non (\M_\kappa) = \kappa^+$
and $\dd_h (\in^*) = \dd_\kappa (\p \in^*) = \cof (\M_\kappa) = \dd_\kappa = \dd_\kappa (\neq^*) = \cov (\M_\kappa) = \lambda =
2^\kappa$.


\subsection{Generalized Hechler forcing}   \label{Hechler}

The generalization of Hechler forcing was first studied by Cummings and Shelah~\cite{CumSh}. 
Assume $\kappa = 2^{< \kappa}$ is a regular uncountable cardinal.

\begin{defi}
{\em Generalized Hechler forcing} $\DD_\kappa$ is defined as follows:
\begin{itemize}
\item conditions are of the form $(s,f)$ where $s \in \kappa^{< \kappa}$, $f \in \kappa^\kappa$ and $s \subseteq f$;
\item the order is given by $(t,g) \leq (s,f)$ if $t \supseteq s$ and $g$ dominates $f$ everywhere, that is, $f (\alpha) \leq
   g(\alpha)$ for all $\alpha < \kappa$.
\end{itemize}
\end{defi}

$\mathbb{D}_\kappa$ generically adds a function $d = d_G$ from $\kappa$ to $\kappa$, called {\em $\kappa$-Hechler function}
and given by $d = \bigcup \{s: $  $\exists f \in \kappa^\kappa 
\; ( (s,f) \in G) \}$, where $G$ is a $\mathbb{D}_\kappa$-generic filter. Clearly $d$ eventually dominates all functions in 
$\kappa^\kappa \cap V$~\cite[Lemma 7]{CumSh}. 

Also it is easy to see that $\mathbb{D}_\kappa$ adds a $\kappa$-Cohen function $c_d \in 2^\ka$, defined by $c_d(\al) = d(\al)\mod 2$ for $\alpha < \kappa$.

Finally $\mathbb{D}_\kappa$ is $\kappa$-centered, and thus has the $\kappa^+$-cc, and also is $< \kappa$-closed~\cite[Lemma 7]{CumSh}.
Let $\lambda \geq \kappa^+$ be a regular cardinal.
As in the countable case, iterate $\mathbb{D}_\kappa$ with supports of size $< \kappa$ for $\lambda$ many steps.
The iteration $\PP_\lambda$ still is $\kappa^+$-cc and (trivially) $< \kappa$-closed. To see the former, first note that conditions $p$ whose 
first coordinates are ground model objects (that is, for all $\alpha \in \supp (p)$, there are $s_\alpha \in \kappa^{< \kappa}$
and a $\PP_\alpha$-name for a function $\dot f_\alpha$ such that $p \re \alpha \forces_\alpha p (\alpha) = (s_\alpha , \dot f_\alpha)$)
are dense and then use a $\Delta$-system argument. Density of such $p$ is a standard argument (see 
Preliminary Lemma~\ref{prelim} below for a similar proof), and we omit the details. In fact, by Lemma~\ref{iteration-centered},
any iteration of $\DD_\kappa$ of length $< (2^\kappa)^+$ still is $\kappa$-centered.

First assume GCH and $\lambda > \kappa^+$. By Lemma~\ref{basic}, all new functions added by $\PP_\lambda$ in $\kappa^\kappa$ already
lie in an intermediate extension. Hence, we are adding $\lambda$ many $\kappa$-Hechler functions that witness 
$\mathfrak{b}_\kappa\geq \lambda$ and since we are also adding $\kappa$-Cohen functions (see also Lemma~\ref{basic2}) we obtain 
$\cov(\mathcal{M}_\kappa)\geq \lambda$. Thus, using the relation $\add(\mathcal{M}_\kappa) \geq
\min \{ \cov(\mathcal{M}_\kappa), \mathfrak{b}_\kappa\}$ (see Corollary~\ref{Truss-cor}),
we conclude that in the generic extension $\add(\mathcal{M}_\kappa) = 2^\kappa = \lambda$.
In case $\kappa$ additionally is inaccessible, we also know that $\bb_\kappa (\in^*) = \kappa^+$ 
by the arguments of Subsection~\ref{total-versus-partial}.
However, we do not know the value of $\bb_\kappa (\p \in^*)$ in the extension. In the classical Hechler
model for $\kappa = \omega$, the corresponding cardinal stays small~\cite[Theorem A]{evpr}.
If this was true for inaccessible $\kappa$, Question~\ref{partial-add-question} would have a positive answer.

Now assume that $2^\kappa \geq \kappa^{++}$ and let $\lambda = \kappa^+$. 
Then the $\lambda$ many $\kappa$-Hechler functions will witness $\dd_\kappa = \kappa^+$, and
the $\kappa$-Cohen functions, $\non (\M_\kappa) = \kappa^+$. Hence $\cof (\M_\kappa) = \kappa^+ < 2^\kappa$ follows
by Corollary~\ref{Truss-cor}. 
In case $\kappa$ additionally is inaccessible, the arguments of Subsection~\ref{total-versus-partial} also
give us $\dd_\kappa (\in^*)  = 2^\kappa$, but we do not know the value of $\dd_\kappa (\p \in^*)$.

The second model is dual to the first in the same way there is duality of models in case $\kappa = \omega$.


We now use generalized Hechler forcing to create a model where $2^{< \kappa} > \kappa$ and the middle part of Cicho\'n's diagram
diagram is split horizontally into three levels.

\begin{thm}   \label{layers}
Assume GCH and let $\kappa$ be a regular uncountable cardinal. There is a cofinality-preserving generic extension in which
$\add (\M_\kappa) = \cov (\M_\kappa) = \kappa^+$, $\bb_\kappa = \bb_\kappa (\neq^*) = \dd_\kappa = \dd_\kappa (\neq^*) =
\kappa^{++}$, and $\non (\M_\kappa) = \cof (\M_\kappa) = 2^\omega = 2^\kappa = \kappa^{+++}$.
\end{thm}

In case $\kappa$ is strongly inaccessible in the ground model -- and weakly inaccessible in the extension -- the
consistency of $ \cov (\M_\kappa) <  \dd_\kappa (\neq^*)$ answers
a question of Matet and Shelah~\cite[Section 4]{MSta} (see also~\cite[Question 3.8, part 2]{KLLSta}).

\begin{proof}
First add $\kappa^{++}$ many $\kappa$-Hechler functions in an iteration with supports of size $< \kappa$. By an earlier comment,
$\add (\M_\kappa) = \bb_\kappa = \dd_\kappa = \cof (\M_\kappa) = \kappa^{++}  = 2^\kappa$ holds in the generic extension. A fortiori, all cardinals
mentioned in the statement of the theorem will be equal to $\kappa^{++}$. Also, $2^{< \kappa} = \kappa$ still holds.

Assume $\kappa \geq \omega_2$. 
Partition $\kappa$ into intervals $J_\alpha$, $\alpha < \kappa$, such that each $J_\alpha$ has size at least $\omega_1$ and $< \kappa$.
Consider the space $\bar \X$ of functions $\bar f$ such that $\dom (\bar f) = \kappa$ and $\bar f (\alpha) \in \kappa^{J_\alpha}$
for all $\alpha < \kappa$. Since $|\kappa^{J_\alpha}| = \kappa^{< \kappa} = \kappa$, $\bar \X$ can be identified with $\kappa^\kappa$.
In particular, since $\dd_\kappa (\neq^*) = \kappa^{++}$, we see that whenever $\bar \F \sub \bar \X$ is of size $\leq \kappa^+$,
then there is $\bar g \in \bar \X$ such that for all $\bar f \in \bar \F$, $\bar g (\alpha) = \bar f (\alpha)$ holds for cofinally many
$\alpha < \kappa$. 

The case $\kappa = \omega_1$ is a little more complicated. We consider all possible interval partitions $J$ such that all $J_\alpha$
are countable. Clearly, there are $\omega_2$ of them. We then let $\bar \X^J$ as above. 

Now add $\kappa^{+++}$ Cohen reals. By Observation~\ref{nwd-kappa} we see that $\add (\M_\kappa) = \cov (\M_\kappa) =
\kappa^+$ and $\non (\M_\kappa) = \cof (\M_\kappa) = 2^\omega = 2^\kappa = \kappa^{+++}$. Also, by the ccc-ness, every new
function in $\kappa^\kappa$ is bounded by a function from the intermediate extension. This means that $\bb_\kappa$ and $\dd_\kappa$
are preserved, and their values are still $\kappa^{++}$ in the generic extension.

The main part of the argument is to show that $\bb_\kappa (\neq^*) = \dd_\kappa (\neq^*) = \kappa^{++}$ in the final extension.
By Observation~\ref{ed-obs}, it suffices to prove $\dd_\kappa (\neq^*) \geq \kappa^{++}$ and $\bb_\kappa (\neq^*) \leq \kappa^{++}$.
Work in the intermediate extension.

For the former, let $\dot \F = \{ \dot f_\beta : \beta < \kappa^+ \}$ be a family of names for functions in $\kappa^\kappa$.
For each $\beta < \kappa^+$, recursively produce a function $f_\beta \in \kappa^\kappa$, 
an interval partition $I^\beta = ( I^\beta_\alpha = [ i_\alpha^\beta ,
i_{\alpha + 1}^\beta ) : \alpha < \kappa )$, and, for each $\alpha < \kappa$, maximal antichains $\{ p^\beta_{\alpha,\gamma} :
\gamma \in I^\beta_\alpha \}$ such that all $I^\beta_\alpha$ are countable and $p_{\alpha,\gamma}^\beta \forces
\dot f_\beta (\gamma) = f_\beta (\gamma)$.  This is clearly possible by the ccc. 
 
If $\kappa \geq \omega_2$, simply let $\bar f_\beta$ be the function defined by $\bar f_\beta (\alpha) = f_\beta \re J_\alpha$
for all $\alpha < \kappa$. By the above, there is $\bar g \in \bar \X$ such that for all $\beta < \kappa^+$, $\bar g (\alpha) = \bar f_\beta (\alpha)$ 
holds for cofinally many $\alpha < \kappa$. Define $g \in \kappa^\kappa$ by $g (\gamma) = \bar g (\alpha) (\gamma)$ if 
$\gamma \in J_\alpha$, for $\alpha < \kappa$. We claim that $g$ is forced to agree with all $\dot f_\beta$ cofinally often.
To see this, fix $\beta < \kappa^+$. Also fix some $\gamma_0$. Let $\alpha \geq \gamma_0$ be such that
$\bar g (\alpha) = \bar f_\beta (\alpha)$. Notice that $J_\alpha$
contains one of the intervals $I^\beta_{\alpha '}$ because $J_\alpha$ is uncountable and the intervals of $I^\beta$ are countable.
Let $p$ be an arbitrary condition. There is $\gamma \in I^\beta_{\alpha '}$ such that $p^\beta_{\alpha ' ,\gamma}$ is
compatible with $p$. Let $q$ be a common extension. Clearly 
\[ q \forces \dot f_\beta (\gamma) = f_\beta (\gamma) = \bar f_\beta (\alpha) (\gamma) = \bar g (\alpha) (\gamma) = g (\gamma),\]
as required.

If $\kappa = \omega_1$, first choose an interval partition $J$ dominating all the interval partitions $I^\beta$, $\beta < \kappa^+$.
This is possible because $\bb_\kappa = \kappa^{++}$, by Proposition~\ref{dom-IP}. Then redo the argument of the preceding
paragraph with this $J$ and the space $\bar \X^J$.

The proof of $\bb_\kappa (\neq^*) \leq \kappa^{++}$ is simpler. Let $\G$ be $\kappa^\kappa$ of the intermediate
extension. Clearly $|\G| = \kappa^{++}$. It suffices to prove $\G$ is a witness for $\bb_\kappa (\neq^*)$ in the final extension.
Let $\dot f$ be a name for a function in $\kappa^\kappa$.
Again use the ccc to recursively produce a function $f \in \kappa^\kappa$, an interval partition $I = ( I_\alpha 
 : \alpha < \kappa )$, and, for each $\alpha < \kappa$, maximal antichains $\{ p_{\alpha,\gamma} :
\gamma \in I_\alpha \}$ such that all $I_\alpha$ are countable and $p_{\alpha,\gamma} \forces
\dot f (\gamma) = f (\gamma)$. Clearly $f \in \G$ and $f$ is forced to agree with $\dot f$ cofinally often.
\end{proof}


\subsection{Generalized localization forcing}   \label{gen-loc}

For this subsection, assume that $\kappa$ is strongly inaccessible.

\begin{defi}
The {\em generalized localization forcing} $\LOCforce_\kappa$ is defined as follows:
\begin{itemize}
\item conditions are of the form $p = (\sigma^p, F^p) = (\sigma, F)$ such that for some ordinal $\gamma = \gamma^p< \kappa$, 
   $\dom (\sigma) = \gamma$, 
   $\sigma (\alpha) \in [\kappa]^{|\alpha|}$ for all $\alpha \in \gamma$, $\dom (F) = \kappa$, and $F(\alpha) \in [\kappa]^{\leq |\gamma|}$
   for $\alpha < \kappa$;
\item the order is given by $q = (\sigma^q,F^q) \leq p = (\sigma^p,F^p)$ if $\sigma^q$ end-extends $\sigma^p$ (i.e., $\sigma^p 
   \sub \sigma^q$), $F^q (\alpha )
   \supseteq F^p (\alpha)$ for all $\alpha \in \kappa$, and $F^p (\alpha) \sub \sigma^q (\alpha)$ for all $\alpha \in \gamma^q \sem \gamma^p$.
\end{itemize}
\end{defi}

$\LOCforce_\kappa$ generically adds a slalom $\varphi = \varphi_G \in \Loc_\kappa$
given by $\varphi = \bigcup \{\sigma: $  $\exists F 
\; ( (\sigma,F) \in G) \}$, where $G$ is a $\LOCforce_\kappa$-generic filter. Clearly $\varphi$ localizes all functions in 
$\kappa^\kappa \cap V$. See~\cite[p. 106]{BJ95} for localization forcing $\LOCforce$ on $\omega$.

\begin{lem}
$\LOCforce_\kappa$ is $\kappa^+$-cc and $<\kappa$-closed.
\end{lem}

\begin{proof}
Fix an ordinal $\gamma < \kappa$ and a function $\sigma$ with $\dom (\sigma) = \gamma$ and $\sigma (\alpha) \in
[\kappa]^{|\alpha|}$ for all $\alpha < \gamma$. Let $\{ p_\beta : \beta < \gamma \}$ be conditions with
$\sigma^{p_\beta} = \sigma$ for all $\beta < \gamma$. Then the $p_\beta$ have a common extension $p$
with $\sigma^p = \sigma$ and $F^p (\alpha) = \bigcup_{\beta < \gamma} F^{p_\beta} (\alpha)$ for all
$\alpha < \kappa$. By $2^{< \kappa} = \kappa$, this shows the $\kappa^+$-cc. The $< \kappa$-closure is straightforward.
\end{proof}

\begin{prop}   \label{slalom-model}
Let $\kappa$ be strongly inaccessible and let $\lambda > \kappa^+$ with $\lambda^\kappa = \lambda$. Then:
\begin{romanenumerate}
\item $\kappa^+ <  \bb_\kappa ( \in^*) = \lambda = 2^\kappa$ holds in a $<\kappa$-closed $\kappa^+$-cc extension.
\item $\kappa^+ = \dd_\kappa ( \in^*) < 2^\kappa = \lambda$ holds in a $<\kappa$-closed $\kappa^+$-cc extension.
\end{romanenumerate}
\end{prop}

\begin{proof}
(i) Perform a $\lambda$-length iteration
$(\PP_\alpha, \dot\QQ_\alpha: \alpha < \lambda)$ with $<\kappa$-support of $\LOCforce_\kappa$. 
The iteration still is $<\kappa$-closed and $\kappa^+$-cc. The argument for the latter is like for Hechler forcing,
see Subsection~\ref{Hechler} (see also the proof of Preliminary Lemma~\ref{prelim}). By Lemma~\ref{basic} and genericity,
we see that $2^\kappa = \bb_\kappa (\in^*) = \lambda$ in the resulting model.

(ii) Assume $2^\kappa = \lambda$ in the ground model. Perform an iteration $(\PP_\alpha , \dot \QQ_\alpha :
\alpha < \kappa^+ )$ with $<\kappa$-support of $\LOCforce_\kappa$ of length $\kappa^+$. 
By Lemma~\ref{basic} and genericity, we see that $\dd_\kappa ( \in^*) = \kappa^+$
in the resulting model.
\end{proof}

\begin{prop}   \label{layers2}
Assume GCH and let $\kappa$ be a strongly inaccessible cardinal. There is a cofinality-preserving generic extension in which
$\add (\M_\kappa) = \cov (\M_\kappa) = \kappa^+$, $\bb_\kappa (\in^*) = \dd_\kappa (\in^*) =
\kappa^{++}$, and $\non (\M_\kappa) = \cof (\M_\kappa) = 2^\omega = 2^\kappa = \kappa^{+++}$.
\end{prop}

This shows that the generalization of the Bartoszy\'nski-Raisonnier-Stern Theorem (Theorem~\ref{BRT-thm})
may fail for weakly inaccessible $\kappa$.

\begin{proof}
The argument is similar to the proof of Theorem~\ref{layers}. 
First add $\kappa^{++}$ many $\LOCforce_\kappa$ generics in an iteration with supports of size $< \kappa$. By Proposition~\ref{slalom-model},
$\bb_\kappa (\in^*) = \dd_\kappa (\in^*) =  \kappa^{++}  = 2^\kappa$ holds in the generic extension. Note that the
family $\Phi \sub \Loc_\kappa$ witnessing the value of $\dd_\kappa (\in^*)$ has the property that for all 
$F : \kappa \to [\kappa]^\omega$ there is $\varphi \in \Phi$ such that for $F (\alpha) \sub \varphi (\alpha)$
for all but less than $\kappa$ many $\alpha$.

Now add $\kappa^{+++}$ Cohen reals. By Observation~\ref{nwd-kappa} we see that $\add (\M_\kappa) = \cov (\M_\kappa) =
\kappa^+$ and $\non (\M_\kappa) = \cof (\M_\kappa) = 2^\omega = 2^\kappa = \kappa^{+++}$. Also, by the ccc-ness, for every new
function $f$ in $\kappa^\kappa$, there is a function $F : \kappa \to [\kappa]^\omega$ in the intermediate extension
such that $f(\alpha) \in F(\alpha)$ for all $\alpha < \kappa$. By the previous paragraph, this easily entails that $\bb_\kappa (\in^*)$ and
$\dd_\kappa (\in^*)$ are preserved, and their values are still $\kappa^{++}$ in the generic extension.
\end{proof}


\subsection{Total slaloms versus partial slaloms}   \label{total-versus-partial}

Let $\kappa$ be an uncountable regular cardinal.

\begin{defi}
Assume $\PP$ is $<\kappa$-closed and $\kappa$-centered, say $\PP = \bigcup_{\gamma < \kappa} P_\gamma$ where all
$P_\gamma$ are $< \kappa$-centered. Say that $\PP$ is {\em $\kappa$-centered with canonical lower bounds}
if there is a function $f = f^\PP : \kappa^{<\kappa} \to \kappa$ such that whenever $\lambda < \kappa$ and  $(p_\alpha : \alpha < \lambda )$
is a decreasing sequence with $p_\alpha \in P_{\gamma_\alpha}$, then there is $p \in P_\gamma$ with $p \leq p_\alpha$
for all $\alpha < \lambda$ and $\gamma = f(\gamma_\alpha : \alpha < \lambda)$.
\end{defi}

\begin{lem}  \label{iteration-centered}
Let $\kappa$ be an uncountable regular cardinal and assume $2^{<\kappa} = \kappa$. Let $\mu < (2^\kappa)^+$ be an ordinal.
Assume $(\PP_\alpha, \dot \QQ_\alpha : \alpha < \mu)$ is an iteration with $<\kappa$-support of $<\kappa$-closed,
$\kappa$-centered forcing notions with canonical lower bounds such that the functions $f^{\dot \QQ_\alpha}$ witnessing canonical
lower bounds lie in the ground model. Then $\PP_\mu$ is $<\kappa$-closed and $\kappa$-centered.
\end{lem}

\begin{proof}
It is well-known that such iterations are $<\kappa$-closed. So let us prove they are $\kappa$-centered.
Assume $\dot \QQ_\alpha = \bigcup_{\gamma < \kappa} \dot Q_{\alpha,\gamma}$ is (forced to be by the trivial condition)
a decomposition of $\dot \QQ_\alpha$ into $<\kappa$-centered sets such that the function $f^{\dot \QQ_\alpha}$ giving
canonical lower bounds associated with this partition belongs to the ground model.
We start with:

\begin{prelem} (For arbitrary $\mu$.)  \label{prelim}
Conditions $p \in \PP_\mu$ such that for all $\beta \in \supp (p)$ there is $\gamma < \kappa$ with 
$p \re \beta \forces p(\beta) \in \dot  Q_{\beta,\gamma}$ are dense.
\end{prelem}

\begin{proof}
Fix $p$. First we construct $q \leq p$ such that for all $\beta \in \supp (p)$ there is $\gamma < \kappa$ with
$q \re \beta \forces q(\beta) \in \dot  Q_{\beta,\gamma}$. Let $\lambda : = |\supp (p)| < \kappa$. Enumerate
$\supp (p)$ as $(\beta_\delta: \delta < \lambda)$ such that each $\beta \in \supp (p)$ appears cofinally often
in the enumeration. Construct a decreasing chain of conditions $(q^\delta:\delta < \lambda)$ such that
$q^0 = p$ and $q^{\delta + 1} \re \beta_{\delta} \forces q^{\delta + 1} (\beta_\delta ) \in \dot  Q_{\beta_\delta,\gamma_\delta}$
for some $\gamma_\delta < \kappa$. This can be done easily using the closure properties of the iteration.

Let $\supp (q) = \bigcup_{\delta < \lambda} \supp (q^\delta)$ and define $q$ by recursion on $\supp (q)$ as follows:
assume $\beta \in \supp (q)$ and $q \re \beta$ has been defined. If $\beta \notin \supp (p)$, obtain $q(\beta)$,
using $<\kappa$-closure of $\dot \QQ_\beta$, such that $q\re \beta \forces ``q(\beta)$ is a lower bound of
$(q^\delta (\beta) : \delta < \lambda)$". If $\beta \in \supp (p)$, let $\lambda_\beta = 
f^{\dot \QQ_\beta} ( \gamma_\delta : \delta < \lambda$ and $\beta_\delta = \beta)$. Since $q \re \beta$
is a lower bound of $(q^\delta \re \beta : \delta < \lambda)$ and since $q^{\delta+1} \re \beta$ forces
$q^{\delta+1} (\beta) \in \dot  Q_{\beta,\gamma_\delta}$ whenever $\beta_\delta = \beta$, canonical lower bounds
gives us that $q \re \beta$ forces that $(q^{\delta +1} (\beta) : \delta < \lambda$ and $\beta_\delta = \beta)$ has a lower bound $q(\beta) \in 
\dot  Q_{\beta,\lambda_\beta}$. Hence $q$ is as required.

Thus we can construct a decreasing chain $(p^n : n \in \omega)$ such that $p^0 = p$ and for all $n$ and
all $\beta \in \supp (p^n)$ there is $\gamma_\beta^n < \kappa$ with $p^{n+1} \re \beta \forces p^{n+1} (\beta) \in
\dot  Q_{\beta,\gamma_\beta^n}$. As in the previous paragraph we obtain a condition $p^\omega$ below
$(p^n : n \in\omega)$ belonging to the dense set of the preliminary lemma:
$\supp (p^\omega) = \bigcup_n \supp (p^n)$ and for all $\beta \in \supp (p^\omega)$
we define by recursion $p^\omega (\beta)$ as follows: suppose $p^\omega \re \beta$ has been defined.
Let $\gamma_\beta = f^{\dot\QQ_\beta} ( \gamma_\beta^n : n \geq \min \{ k : \beta \in \supp (p^k) \} )$.
Again we see that $p^\omega \re \beta$ forces that $(p^n (\beta) : n \geq \min \{ k : \beta \in \supp (p^k) \} )$
has a lower bound $p^\omega (\beta) \in \dot Q_{\beta, \gamma_\beta}$.
\end{proof}

Given the preliminary lemma, the proof of $\kappa,$-centeredness is a standard argument:
since $\mu < (2^\kappa)^+$, there is an injection $\mu \to 2^\kappa : \alpha \mapsto f_\alpha$. Let $\F$ be the 
collection of all functions $F$ whose domain is a subset of $2^\delta$ for some $\delta = \delta_F < \kappa$ 
of size $<\kappa$ and whose range is a subset of $\kappa$. (If $\kappa$ is inaccessible it suffices to consider
$F$ whose domain is $2^\delta$ for $\delta = \delta_F$ because $2^\delta < \kappa$ then holds.) 
Since $2^{<\kappa} = \kappa$, we have $|\F| = \kappa$. For $F \in \F$ let
\[ P_F = \{ p \in \PP_\mu : \forall \beta \in \supp (p) \; (f_\beta \re \delta_F \in \dom (F) \mbox{ and } p \re \beta \forces
p(\beta) \in \dot Q_{\beta,F(f_\beta \re \delta_F)} )\} \]
We need to verify that all $P_F$ are $<\kappa$-centered and that the union of the $P_F$ is dense in $\PP_\mu$.

For the former, assume $\lambda < \kappa$ and let $\{ p_\zeta : \zeta < \lambda \} \sub P_F$. By recursion on $\beta$
construct a common extension $p \re \beta$ of the $p_\zeta \re \beta$. If $\beta$ is limit, simply let $p \re \beta$ be
the union of the $p \re \alpha$ for $\alpha < \beta$. So assume $\beta$ is successor. If $\beta$ belongs to the support 
of some $p_\zeta$ we have $f_\beta \re \delta_F \in \dom (F)$. Now note that since $\dot Q_{\beta, F (f_\beta \re \delta_F)}$
is forced to be $<\kappa$-centered and since $p\re\beta$ forces $p_\zeta (\beta) \in \dot Q_{\beta, F(f_\beta \re \delta_F)}$
or $p_\zeta (\beta) = \one$ for all $\zeta < \lambda$, there is a $\PP_\beta$-name $p(\beta)$ such that $p \re \beta$
forces $p(\beta) \leq p_\zeta (\beta)$ for all $\zeta < \lambda$. If $\beta$ does not belong to the support of any $p_\zeta$,
let $p(\beta) = \one$. This defines $p \re \beta + 1$. By construction $|\supp (p) | < \kappa$ and, thus, $p \in \PP_\mu$.

For the latter, we use the preliminary lemma. Let $p \in \PP_\mu$ be a condition such that for all $\beta \in \supp (p)$,
there is $\gamma_\beta < \kappa$ with $p\re\beta \forces p(\beta) \in \dot Q_{\beta , \gamma_\beta}$. Since
$|\supp (p) | < \kappa$ we can find $\delta < \kappa$ such that $f_\beta \re \delta$, $\beta \in \supp (p)$, are all distinct.
Let $\delta_F = \delta$ and let $F \in \F$ be the function with domain $\{ f_\beta \re \delta : \beta \in \supp (p) \}$
such that for all $\beta \in \supp (p)$, $F(f_\beta \re \delta) = \gamma_\beta$. Then $p \in P_F$ is immediate.
\end{proof}

Note that this is optimal. It is well-known (and easy to see) that $<\kappa$-support iterations of length $\geq (2^\kappa)^+$
of non-trivial forcing notions never are $\kappa$-centered.


For the rest of this subsection, assume $\kappa$ is strongly inaccessible. 

\begin{defi}
The {\em generalized partial localization forcing} $\PLOCforce_\kappa$ is defined as follows:
\begin{itemize}
\item conditions are of the form $p = (\sigma^p, F^p) = (\sigma, F)$ such that $\dom (\sigma) \sub \kappa$, $|\dom (\sigma) | < \kappa$,
   $\sigma (\alpha) \in [\kappa]^{|\alpha|}$ for all $\alpha \in \dom (\sigma)$, $\dom (F) = \kappa$, and $F(\alpha) \in [\kappa]^\lambda$
   for all $\alpha < \kappa$, for some fixed $\lambda < \kappa$;
\item the order is given by $q = (\sigma^q,F^q) \leq p = (\sigma^p,F^p)$ if $\sigma^q$ end-extends $\sigma^p$ (i.e., $\sigma^p 
   \sub \sigma^q$ and $\alpha \in \dom (\sigma^q \sem \sigma^p)$ implies $\alpha \geq \sup (\dom (\sigma^p))$), $F^q (\alpha )
   \supseteq F^p (\alpha)$ for all $\alpha \in \kappa$, and $F^p (\alpha) \sub \sigma^q (\alpha)$ for all $\alpha \in \dom (\sigma^q \sem \sigma^p)$.
\end{itemize}
\end{defi}

$\PLOCforce_\kappa$ generically adds a partial slalom $\varphi = \varphi_G \in \pLoc_\kappa$
given by $\varphi = \bigcup \{\sigma: $  $\exists F 
\; ( (\sigma,F) \in G) \}$, where $G$ is a $\PLOCforce_\kappa$-generic filter. Clearly $\varphi$ localizes all functions in 
$\kappa^\kappa \cap V$. See~\cite[p. 47]{Br06} for partial localization forcing $\PLOCforce$ on $\omega$.

\begin{lem}  \label{Pp-centered}
Assume $\kappa$ is strongly inaccessible.
Then $\PLOCforce_\kappa$ is $<\kappa$-closed and $\kappa$-centered with canonical lower bounds.
Furthermore, if $V \sub W$ are ZFC-models such that $2^{<\kappa} \cap V = 2^{< \kappa} \cap W$ and 
$\PLOCforce \in W$, then the function $f$ witnessing canonical lower bounds may be taken in $V$.
\end{lem}

\begin{proof}
For $\sigma \in \kappa^{<\kappa}$, let $P_\sigma = \{ (\sigma,F) : (\sigma,F) \in \PLOCforce_\kappa \}$. Clearly, $P_\sigma$ is $< \kappa$-centered and
$\PLOCforce_\kappa = \bigcup_\sigma P_\sigma$. Also given $\lambda < \kappa$ and $(p_\alpha \in P_{\sigma_\alpha}: \alpha < \lambda)$ decreasing,
we necessarily have that $\alpha < \beta$ implies $\sigma_\alpha \sub \sigma_\beta$ and there is a lower bound $p \in P_\sigma$
where $\sigma = \bigcup_{\alpha < \lambda} \sigma_\alpha$. Hence $f : (\kappa^{<\kappa})^{<\kappa} \to \kappa^{<\kappa}$ such that
$f (\sigma_\alpha : \alpha < \lambda) = \bigcup_{\alpha < \lambda} \sigma_\alpha$ for increasing sequences $(\sigma_\alpha : \alpha < \lambda)$
witnesses canonical lower bounds. $f \in V$ in the furthermore clause is immediate.
\end{proof}

\begin{lem} \label{slalom-preservation}
Let $\kappa$ be strongly inaccessible and
let $\PP$ be a $\kappa$-centered forcing notion. Let $h \in \kappa^\kappa$ and assume $\dot \varphi$ is a
$\PP$-name for an $h$-slalom.
Then there are $h$-slaloms $(\varphi_\alpha : \alpha < \kappa)$ such that whenever $f \in\kappa^\kappa$ is not localized by
any $\varphi_\alpha$, then $\forces ``\dot\varphi$ does not localize $f"$.
\end{lem}

\begin{proof}
Let $\PP = \bigcup_{\alpha < \kappa} P_\alpha$ where all $P_\alpha$ are $<\kappa$-centered. Let $\dot\varphi$ be a $\PP$-name for
an $h$-slalom. Fix $\alpha < \kappa$. Define $\varphi_\alpha$ as follows:
\[ \varphi_\alpha (\beta) = \{ \gamma < \kappa : \exists p \in P_\alpha \; (p \forces \gamma \in \dot \varphi (\beta) ) \}. \]
Note that $|\varphi_\alpha (\beta) | \leq |h(\beta)|$. (For suppose that $|\varphi_\alpha (\beta)| > |h(\beta)|$. For each $\gamma
\in \varphi_\alpha (\beta)$ find $p_\gamma \in P_\alpha$ such that $p_\gamma \forces \gamma \in \dot\varphi (\beta)$. 
By $<\kappa$-centeredness of $P_\alpha$, we can find $\Gamma \sub \varphi_\alpha (\beta)$ of size $< \kappa$ but larger
than $|h(\beta)|$ and a condition $p$ which extends all $p_\gamma$ with $\gamma \in \Gamma$. Then $p$
forces $|\dot\varphi(\beta)| > |h(\beta)|$, a contradiction.)

Hence all $\varphi_\alpha$ are $h$-slaloms. Fix $f \in \kappa^\kappa$ such that for all $\alpha < \kappa$,
$\exists^\infty \beta \; (f(\beta) \notin \varphi_\alpha (\beta))$. Also fix $p \in \PP$ and $\beta_0 < \kappa$.
There is $\alpha< \kappa$ such that $p \in P_\alpha$. Also there is $\beta \geq \beta_0$ such that
$f (\beta) \notin \varphi_\alpha (\beta)$. In particular, $p$ does not force that $f(\beta)$ belongs to
$\dot \varphi (\beta)$. Hence there is $q\leq p$ such that $q \forces f(\beta) \notin \dot\varphi (\beta)$.
Since $p$ and $\beta_0$ were arbitrary, this shows that the trivial condition forces $\exists^\infty\beta \;
(f(\beta) \notin \dot\varphi (\beta))$.
\end{proof}

\begin{thm}   \label{total-partial-thm}
Let $\kappa$ be strongly inaccessible and let $\lambda > \kappa^+$ with $\lambda^\kappa = \lambda$. Then:
\begin{romanenumerate}
\item $\kappa^+ = \bb_h (\in^*) < \bb_\kappa (\p \in^*) =\lambda =2^\kappa$ (for all $h \in \kappa^\kappa$) holds in a 
   $<\kappa$-closed $\kappa^+$-cc extension.
\item $\kappa^+ = \dd_\kappa (\p \in^*) < \dd_h (\in^*) = \lambda = 2^\kappa$ (for all $h \in \kappa^\kappa$) holds in a 
   $<\kappa$-closed $\kappa^+$-cc extension.
\end{romanenumerate}
\end{thm}

\begin{proof} (i) Assume $2^\kappa = \lambda$ in the ground model $V$. Perform a $\lambda$-length iteration
$(\PP_\alpha, \dot\QQ_\alpha: \alpha < \lambda)$ with $<\kappa$-support of $\PLOCforce_\kappa$. In the resulting model
$2^\kappa = \bb_\kappa (\p \in^*) = \lambda$ be genericity. So it suffices to show that $\bb_h (\in^*) = \kappa^+$
for any $h \in \kappa^\kappa$.

Any such $h$ is added by an initial segment of the iteration (see Lemma~\ref{basic}), say by $\PP_{\alpha_0}$.
Since we iterate with $< \kappa$-support, $\kappa$-Cohen functions are added in $\kappa$-limits of the
iteration. Let $\alpha_1 = \alpha_0 +  \kappa^+$, and let $\{ c_\gamma : \gamma < \kappa^+ \}$
be the $\kappa^+$ many $\kappa$-Cohen functions added between $\alpha_0$ and $\alpha_1$. 
Let $W$ be the $\PP_{\alpha_1}$-generic extension. Clearly (since they are added cofinally),
$\{ c_\gamma : \gamma < \kappa^+ \}$ witnesses $\bb_h (\in^*) = \kappa^+$ in $W$. In fact,
any $h$-slalom in $W$ can localize at most $\kappa$ many of the $c_\gamma$. We need to show
that $\{ c_\gamma : \gamma < \kappa^+ \}$ still witnesses $\bb_h (\in^*) = \kappa^+$ in the final extension. 

Let $\dot \varphi $ be a $\PP_{\lambda}$-name for an $h$-slalom.  Since $2^\kappa = \lambda$,
by Lemmata~\ref{Pp-centered} and~\ref{iteration-centered}, $\PP_\lambda$ is $\kappa
$-centered. Hence we find $(\varphi_\beta : \beta < \kappa )$ in $W$ as in Lemma~\ref{slalom-preservation}.
In $W$, there is $\gamma \in \kappa^+$ such that $c_\gamma$ is not localized by any $\varphi_\beta$. 
Hence, by~Lemma~\ref{slalom-preservation}, $\forces ``\dot\varphi$ does not localize $c_\ga"$, as required.

(ii) Assume $2^\kappa = \lambda$ in the ground model. Perform an iteration $(\PP_\alpha , \dot \QQ_\alpha :
\alpha < \lambda \cdot \kappa^+ )$ with $<\kappa$-support of $\PLOCforce_\kappa$ of length the ordinal $\lambda \cdot \kappa^+$. 
By~Lemmata~\ref{Pp-centered} and~\ref{iteration-centered}, the whole iteration $\PP_{\lambda \cdot \kappa^+}$ is
$\kappa$-centered. In the resulting model, $\dd_\kappa {(\p \in^*)} = \kappa^+$ because the cofinality
of the length of the iteration is $\kappa^+$. Also $2^\kappa = \lambda$ is clear. We need to show that
$\dd_h (\in^*) \geq \lambda$ for any $h \in \ka^\ka$.

Fix $h \in \kappa^\kappa$. Since $h$ arises in an intermediate extension and since all larger intermediate extensions 
by some $\PP_{\lambda \cdot \alpha}$, where $\alpha < \kappa^+$ is a successor ordinal, satisfy $\dd_h (\in^*) =\lambda$ because
of the $\kappa$-Cohen functions added in limit stages of the $\lambda$
preceding steps of the iteration, we may as well assume that $h$ belongs to the ground model
and that $\dd_h (\in^*) = \lambda$ holds there. Now let $\mu < \lambda$ and let $(\dot\varphi^\delta : \delta < \mu )$ be $\PP_{\lambda
\cdot \kappa^+}$-names
for $h$-slaloms. By Lemma~\ref{slalom-preservation}, we obtain $(\varphi^\delta_\beta : \beta < \kappa)$ for each $\delta
< \mu$. By assumption on the ground model, there is $f \in \kappa^\kappa$ which is not localized by any $\varphi^\delta_\beta$.
Therefore, by Lemma~\ref{slalom-preservation}, $\forces ``$no $\dot\varphi^\delta$ localizes $f"$. In particular, the $(\dot\varphi^\delta:
\delta < \mu)$ do not form a witness for $\dd_h (\in^*) = \mu$ in the generic extension, and $\dd_h (\in^*) \geq
\lambda$ follows.
\end{proof}



\section{Iterations and products with support of size $\kappa$}

Here we consider iterations and products with support of size $\kappa$ of forcing notions which are $< \kappa$-closed, preserve $\kappa^+$ and are
$\kappa^{++}$-cc. Such constructions correspond to countable support iterations and products of proper forcing notions in case $\kappa = \omega$.
The latter are a powerful tool for separating the cardinal invariants in Cicho\'n's diagram; in fact, any consistent partition of the diagram
into two parts, one assuming the value $\aleph_1$ and the other, $\aleph_2$, can be obtained by such an iteration. The main tools for
proving this are preservation theorems, saying roughly that if a forcing has a certain property, then so does its countable support
iteration. There are no such preservation theorems for uncountable $\kappa$, and the only feasible approach as of now is to prove
directly that the whole iteration has the property in question. The main tool for this is to use generalized fusion. While this can
be done for generalized Sacks forcing (see the discussion at the beginning of 5.1.2), we do not know whether 
generalized fusion can be used for generalized Miller forcing to show the (appropriate version of the) generalized Laver property
(see Question~\ref{Miller-Laver} below).

In this section, $\kappa$ is always a strongly inaccessible cardinal.


\subsection{Generalized Sacks forcing}

\subsubsection{The single step}

The generalization of Sacks forcing was first studied by Kanamori~\cite{Ka80}. 
$T \sub 2^{< \kappa}$ is a {\em tree} if it is closed under initial segments, that is, $u \in T$ and $v \sub u$ imply $v \in T$.
A node $u\in T$ {\em splits} in $T$ if both $u\ha 0$ and $u \ha 1$ belong to $T$. 

\begin{defi}
{\em Generalized Sacks forcing $\SS_\kappa$} is defined as follows:
\begin{itemize}
\item conditions in $\mathbb{S}_\kappa$ are subtrees $T$ of $2^{<\kappa}$ such that:
\begin{enumerate} 
\item each $u \in T$ has a splitting extension  $t \in T$, that is, $u \sub t$ and $t$ splits in $T$,
\item if $\alpha < \kappa$ is a limit ordinal, $u \in 2^\alpha$, and $u \re \beta \in T$ for every $\beta < \alpha$, then $u \in T$,
\item if $\alpha < \kappa$ is a limit ordinal, $u \in 2^\alpha$, and for arbitrarily large $\beta < \alpha$, $u \restriction \beta$ splits in $T$, then $u$ splits in $T$;
\end{enumerate}
\item the order is given by inclusion, i.e. $T \leq S$ if $T \subseteq S$. 
\end{itemize}
\end{defi}

In fact, $\SS_\kappa$ can be defined for any regular $\kappa$ with $2^{< \kappa} = \kappa$. However, many
of the typical properties of Sacks forcing in the countable case, like the $\omega^\omega$-bounding property or
the stronger  Sacks property (see Proposition~\ref{S-Sacksprop} below) only generalize when $\kappa$ is strongly inaccessible 
(see the discussion after Theorem 1.5 in~\cite{Ka80} for details), and we therefore consider only the inaccessible case.

$\mathbb{S}_\kappa$ generically adds a function $s = s_G$ from $\kappa$ to $2$, called a {\em $\kappa$-Sacks function}
and given by $s = \bigcap \{T: T \in G \}$, where $G$ is an $\mathbb{S}_\kappa$-generic filter. 

Given $T \in \SS_\kappa$ and $u \in T$, let $T_u = \{ t \in T : t \sub u$ or $u \sub t\}$ be the {\em subtree} of $T$ given by $u$.
For $T \in \mathbb{S}_\kappa$, the {\em stem of $T$}, $\stem(T)$, is the unique splitting node that is comparable with all elements in $T$. 

\begin{defi}[The $\alpha$-th splitting level of $T$]
Given $T \in \mathbb{S}_\kappa$, define by recursion on $\ka$:
\begin{itemize}
\item $\Split_0 (T) = \{ \stem(T) \}$,
\item $\Split_{\alpha+1}(T) = \{ \stem(T_{u \ha i}) : u \in \Split_\alpha(T)$ and $i \in 2 \}$,
\item if $\alpha$ is a limit ordinal $<\ka$, $\Split_{\alpha} (T) = \{ u \in T: $ there is an increasing sequence $(u_\beta \in \Split_\beta(T):
\beta < \alpha)$ such that $u = \bigcup_{\beta< \alpha} u_\beta \}$.  
\end{itemize}
\end{defi} 

Since there is a canonical bijection $b$ between $2^{<\kappa}$ and $\bigcup_{\alpha < \kappa} \Split_\alpha (T)$ sending
elements of $2^\alpha$ to $\Split_\alpha (T)$ and recursively defined by 
\begin{itemize}
\item $b ( \emptyset ) = \stem (T)$,
\item $b ( u \ha i ) = \stem (T_{ b(u) \ha i})$ for $u \in 2^\alpha$ and $i \in 2$,
\item $b (u) = \bigcup_{\beta < \alpha} b (u \re \beta )$ if $\alpha$ is a limit ordinal and $u \in 2^\alpha$,
\end{itemize}
we see immediately that $|  \Split_\alpha (T) | = | 2^\alpha | = 2^{|\alpha|}$. Using the splitting levels, define the
{\em fusion orderings} $\leq_\alpha$ for $\alpha < \kappa$ by $S \leq_\alpha T$ if $S \leq T$ and $\Split_\al(T)= \Split_\al(S)$. 

\begin{defi}[Fusion sequence]
A sequence of conditions $(S_\al: \al< \ka) \subseteq \mathbb{S}_\kappa$ is a {\em fusion sequence} if 
\begin{itemize}
\item $S_{\alpha +1} \leq_\al S_\al$ for every $\alpha < \kappa$, 
\item $S_\delta= \bigcap_{\al< \delta} S_\al$ for limit $\delta< \ka$.
\end{itemize}
\end{defi}

The Fusion Lemma~\cite[Lemma 1.4]{Ka80} says that
for a fusion sequence $(S_\al: \al< \ka)$, the {\em fusion} $S = \bigcap_{\alpha < \kappa} S_\alpha$ is a condition in $\SS_\kappa$. This is used
to show preservation of $\kappa^+$~\cite[Theorem 1.5]{Ka80}; we shall prove a stronger statement in Proposition~\ref{S-Sacksprop}
below. Also, preservation of cardinals up to and including $\kappa$ follows from the easy $< \kappa$-closure of the 
forcing~\cite[Lemma 1.2]{Ka80}.

\begin{defi}[Generalized Sacks property]
Let $h \in \kappa^\kappa$ with $\lim_{\alpha \to \kappa} h(\alpha) = \kappa$. 
A forcing notion $\mathbb{P}$ has the {\em generalized $h$-Sacks property} if for every condition $p \in \mathbb{P}$ and every $\mathbb{P}$-name $\dot{f}$ for an element in $\kappa^\kappa$ there are a condition $q \leq p$ and an $h$-slalom $\varphi \in \Loc_h$ such that $q \Vdash \dot{f}(\alpha) \in \varphi (\alpha)$ for all $\alpha < \kappa$. If $h = \id$, we say $\PP$ has the {\em generalized Sacks property}.
\end{defi}

\begin{prop}  \label{S-Sacksprop}
Let $h$ be the power set function, i.e., $h (\alpha) = 2^{|\alpha|}$ for all $\alpha < \kappa$.
Then $\mathbb{S}_\kappa$ has the generalized $h$-Sacks property. In fact, given $T \in \SS_\kappa$ and any $\SS_\kappa$-name
$\dot f : \kappa \to V$ there are $S \leq T$ and $\varphi : \kappa \to V$ such that $|\varphi (\alpha) | \leq 2^{|\alpha|}$ for all
$\alpha < \kappa$ and $S \forces \dot f (\alpha) \in \varphi (\alpha)$ for $\alpha < \kappa$.
\end{prop}

Note that this immediately implies preservation of $\kappa^+$.

\begin{proof}
Clearly it suffices to show that given $T \in \SS_\kappa$ and a sequence $( \A_\alpha : \alpha < \kappa )$ of maximal antichains in $\SS_\kappa$,
there are $S \leq T$ and $\B_\alpha \sub \A_\alpha$ such that for all $\alpha < \kappa$, $|\B_\alpha | \leq 2^{|\alpha|}$ and 
$\B_\alpha$ is predense below $S$. More explicitly, we recursively construct a fusion sequence of conditions $(T_\alpha: \alpha < \kappa)$ 
and sets $\B_\alpha \subseteq \mathcal{A}_\alpha$ satisfying: 
\begin{itemize}
 \item $\B_\alpha $ is predense below $T_{\alpha + 1}$,
 \item $ \lvert \B_\alpha \lvert \leq 2^{ \lvert \alpha + 1 \lvert}$.
\end{itemize}
For the basic step, let $T_0 = T$. 

For the successor step, suppose that $T_\alpha$ has already been constructed. Fix $u \in \Split_\alpha(T_\alpha)$ and $i \in 2$, and consider the subtree $(T_\alpha)_{u \ha i}$. Since $\mathcal{A}_{\alpha}$ is a maximal antichain there exists a condition $U_{u\ha i} \in \mathcal{A}_{\alpha}$ compatible with $(T_\alpha)_{u \ha i}$. Let $(T_{\alpha+1})_{u \ha i}$ be a common extension. Put $T_{\alpha+1}= \bigcup \{ (T_{\alpha + 1})_{u \ha i} : u \in \Split_\alpha(T_\alpha)$ and $i \in 2\}$. Thus $T_{\alpha+1} \leq_\alpha T_\alpha$ holds and we can define $\B_{\alpha}=\{ U_{u \ha i} : u \in \Split_\alpha (T_\alpha)$ and $i \in 2\}$. Clearly this set has size at most $2^{|\alpha+1|}$ and by construction it is predense below $T_{\alpha+1}$.

For the limit step, suppose we already have constructed $(T_\alpha: \alpha < \delta)$ where $\delta$ is a limit ordinal. We let $T_\delta = \bigcap_{\alpha< \delta} T_\alpha$.  

Take the fusion $S = \bigcap_{\alpha < \kappa} T_\alpha$ of the sequence $(T_\alpha: \alpha < \kappa)$. Then $S \leq_{\alpha} T_\alpha$ for all $\alpha$ which implies that for all $\alpha$, $\B_\alpha$ is predense below $S$. Adjusting the construction for the finite ordinals, if necessary, and noting that $|\alpha + 1 | = |\alpha|$ for infinite ordinals $\alpha$, we see that $| \B_\alpha | \leq 2^{ |\alpha|}$ for all $\alpha < \kappa$, as required.
\end{proof}

\begin{prop}  \label{S-notSacksprop}
$\mathbb{S}_\kappa$ does not have the generalized Sacks property. In fact, $\SS_\kappa$ adds a function $f \in \kappa^\kappa$
such that for all $\varphi \in \Loc_\kappa$ from the ground model, $f (\alpha) \notin \varphi (\alpha)$ for cofinally
many $\alpha$.
\end{prop}

\begin{proof}
Let $g$ be a bijection between $\kappa$ and $2^{< \kappa}$. Let $\dot s \in 2^\kappa$ be the name for the generic
$\kappa$-Sacks function. Define the name $\dot f \in \kappa^\kappa$ by $\dot f (\alpha) =  g^{-1} ( \dot s \re \alpha)$ for $\alpha < \kappa$.
Take a slalom $\varphi$ from the ground model and an arbitrary condition $T \in \SS_\kappa$. Fix a cardinal $\alpha_0 < \kappa$.
We need to find $\alpha \geq \alpha_0$ and $S \leq T$ such that $S \forces \dot f (\alpha ) \notin \varphi (\alpha)$.

To this end, recursively construct a strictly increasing sequence of cardinals $(\alpha_n < \kappa : n \geq 1)$ such that 
$\alpha_1 > \alpha_0$ and $\Split_{\alpha_n} (T) \sub 2^{ \leq \alpha_{n+1}}$ for every $n\in\omega$.
Put $\alpha = \sup \{ \alpha_n : n \in \omega \}$. Then $\Split_\alpha (T) = T \cap 2^\alpha$.
In particular $| T \cap 2^\alpha | = | \Split_\alpha (T) | = 2^{ |\alpha |} > |\alpha|$.
Hence we can find $u \in T \cap 2^{\alpha}$ such that
$g^{-1} (u) \notin \varphi (\alpha)$. Since $S = T_u$ forces $\dot s \re \alpha = u$, it also forces $\dot f (\alpha)
\notin \varphi (\alpha)$, as required. 
\end{proof}

Note that the argument in this proof shows that the set $\{ \alpha < \kappa : \Split_\alpha (T) = T \cap 2^\alpha \}$
contains a club. In particular, Proposition~\ref{S-notSacksprop} still holds with the identity $\id$ replaced by any 
$h \in \kappa^\kappa$ such that $\{ \alpha < \kappa : h(\alpha ) < 2^{|\alpha|} \}$ is stationary.


\subsubsection{The product (iteration)}

Propositions~\ref{S-Sacksprop} and~\ref{S-notSacksprop} suggest that by adding many $\kappa$-Sacks functions, 
either by a product or by an iteration, we should be
able to distinguish between the two localization cardinals $\dd_\kappa (\in^*)$ and $\dd_h (\in^*)$ where $h$ is the power
set function. We shall now see that this is indeed the case and, in fact, both the product and the iteration work. Since the
result is slightly more general for the former, allowing for arbitrary values of $2^\kappa$, we shall concentrate on the
product and simply mention that iterating generalized Sacks forcing for $\kappa^{++}$ steps  with supports of size $\kappa$
over a model of GCH yields a model for $\dd_h (\in^*) = \kappa^+ < \dd_\kappa (\in^*) = \kappa^{++} = 2^\kappa$.
See~\cite[Sections 2 and 6]{Ka80} and~\cite{DF10} for such iterations.

\begin{defi}
For a set $A$ of ordinals, $\SS_\kappa^A$ is the {\em $\kappa$-support product} of $\SS_\kappa$ with index set $A$, that is,
$\SS_\kappa^A$ consists of all functions $p : A \to \SS$ such that $\supp (p) = \{ \beta \in A : p (\beta) \neq 2^{<\kappa} \}$
has size at most $\kappa$. $\SS_\kappa^A$ is ordered coordinatewise: $q \leq p$ if $q(\beta) \leq p (\beta)$ for all $\beta \in A$
(this clearly implies $\supp (q) \supseteq \supp (p)$).
\end{defi}

For each $\beta \in A$, $\SS_\kappa^A$ adds a function $s_\beta$ which is $\SS_\kappa$-generic over the ground model.
If $\delta < \kappa$ is a limit ordinal and $(p_\alpha : \alpha < \delta )$ is a decreasing chain of conditions in $\SS_\kappa^A$,
$p= \bigwedge_{\alpha < \delta} p_\alpha \in \SS_\kappa^A$ is defined by $p(\beta) = \bigcap_{\alpha < \delta} p_\alpha (\beta)$. 
In particular, $\SS_\kappa^A$ is $<\kappa$-closed. Furthermore,
if $F \sub A$ has size less than $\kappa$ and $\alpha < \kappa$, we define the {\em fusion ordering} 
$\leq_{F,\alpha}$ by $q \leq_{F,\alpha} p$ if $q \leq p$ and $q (\beta) \leq_\alpha p (\beta)$ for all $\beta \in F$. 

\begin{defi}[Generalized fusion sequence]
A sequence $(p_\al, F_\al : \al< \ka)$ is a {\em generalized fusion sequence} if $p_\al \in \SS_\kappa^A$, $F_\alpha \in [A]^{< \kappa}$, and
\begin{itemize}
\item $p_{\alpha + 1} \leq_{F_\alpha, \alpha} p_\alpha$ for $\alpha < \kappa$, and $p_\delta = \bigwedge_{\alpha<\delta} p_\alpha$ for limit ordinals $\delta$,
\item $F_\alpha \sub F_{\alpha + 1}$ for $\alpha < \kappa$, $F_\delta = \bigcup_{\alpha < \delta} F_\alpha$ for limit ordinals $\delta$, and 
   $\bigcup_{\alpha < \kappa} F_\alpha = \bigcup_{\alpha < \kappa} \supp (p_\alpha)$.
\end{itemize}
\end{defi}

The Generalized Fusion Lemma~\cite[Lemma 1.9 and Section 5]{Ka80} says that the {\em fusion} $p = \bigwedge_{\alpha < \kappa}
p_\alpha$ of a generalized fusion sequence $(p_\al, F_\al : \al< \ka)$ is a condition in $\SS_\kappa^A$. 
This is used to show preservation of $\kappa^+$~\cite[Theorem 5.2]{Ka80}. If $2^\kappa = \kappa^+$ a straightforward
$\Delta$-system argument yields the $\kappa^{++}$-cc~\cite[Theorem 5.3]{Ka80} and, thus, preservation of cardinals $\geq \kappa^{++}$.

\begin{mainlem}   \label{prod-Sacksprop}
Let $h$ be the power set function, i.e., $h (\alpha) = 2^{|\alpha|}$ for all $\alpha < \kappa$.
Then $\SS^A_\kappa$ has the generalized $h$-Sacks property. In fact, given $p \in \SS_\kappa^A$ and any $\SS_\kappa^A$-name
$\dot f : \kappa \to V$ there are $q \leq p$ and $\varphi : \kappa \to V$ such that $|\varphi (\alpha) | \leq 2^{|\alpha|}$ for all
$\alpha < \kappa$ and $q \forces \dot f (\alpha) \in \varphi (\alpha)$ for $\alpha < \kappa$.
\end{mainlem}

Note that this also implies preservation of $\kappa^+$.

\begin{proof}
This is an elaboration of the proof of Proposition~\ref{S-Sacksprop}.  Let $p \in \SS^A_\kappa$ and let
$( \A_\alpha : \alpha < \kappa )$ be a sequence of maximal antichains in $\SS^A_\kappa$. 
We recursively construct a generalized fusion sequence $(p_\al, F_\al : \al< \ka)$ and sets
$\B_\alpha \subseteq \mathcal{A}_\alpha$ satisfying: 
\begin{itemize}
 \item $|F_\alpha| \leq |\alpha|$,
 \item $\B_\alpha $ is predense below $p_{\alpha + 1}$,
 \item $ \lvert \B_\alpha \lvert \leq 2^{ \lvert \alpha \lvert}$ for infinite $\alpha$.
\end{itemize}
For the basic step, let $p_0 = p$ and $F_0 = \emptyset$. 

For the successor step, suppose that $p_\alpha$ and $F_\alpha$ have already been constructed. 
Let $\bar U = \{ \bar u = ( u_\beta : \beta \in F_\alpha ) : u_\beta \in \Split_\alpha (p_\alpha (\beta))$ for all $\beta \in F_\alpha\}$.
Since $| \Split_\alpha (p_\alpha (\beta)) | = 2^{|\alpha|}$ for each $\beta$ and $|F_\alpha| \leq |\alpha|$, we
see that $|\bar U | = (2^{|\alpha|})^{|\alpha|} = 2^{|\alpha|}$ for infinite $\alpha$. Thus let
$( (\bar u_\gamma , g_\gamma) : \gamma < 2^{|\alpha|} )$ enumerate all pairs $(\bar u , g) \in \bar U \times 2^{F_\alpha}$.
Recursively construct a decreasing chain $( q_\alpha^\gamma :  \gamma < 2^{|\alpha|} )$ of conditions and
a sequence $(r_\alpha^\gamma : \gamma < 2^{|\alpha|})$ of elements of $\A_\alpha$ such that
\begin{itemize}
\item $q_\alpha^0 = p_\alpha$, $q_\alpha^\delta \leq_{F_\alpha ,\alpha} q_\alpha^\gamma$ for all $\gamma \leq \delta$,
\item $(q_\alpha^{\gamma + 1})_{(\bar u_\gamma , g_\gamma)} \leq r_\alpha^\gamma$,
\item $q_\alpha^\delta = \bigwedge_{\gamma <\delta} q_\alpha^\gamma$ for limit ordinals $\delta$.
\end{itemize}
Here, for $(\bar u, g) \in \bar U \times 2^{F_\alpha}$, $q_{(\bar u , g)}$ is defined by:
\[ q_{(\bar u , g)} (\beta)=
\begin{cases}
(q (\beta))_{ u_\beta \ha g(\beta)} & \text{if } \beta \in F_\alpha \\
q  (\beta)  & \text{otherwise.}
\end{cases}
\] 
Since the basic step and the limit step are straightforward, it suffices to do the successor step of this recursion. 
Assume $q_\alpha^\gamma$ has been produced. Since $\mathcal{A}_{\alpha}$ is a maximal antichain 
there exists a condition $r_\alpha^\gamma \in \mathcal{A}_{\alpha}$ compatible with $(q_\alpha^{\gamma })_{(\bar u_\gamma , g_\gamma)}$. 
Let $q'$ be a common extension. Define $q_{\alpha}^{\gamma + 1}$ such that
\[ q_\alpha^{\gamma + 1} (\beta) =
\begin{cases}
q' (\beta) \cup \bigcup \{ (q_\alpha^\gamma (\beta) )_{u \ha i} : u \in \Split_\alpha (p_\alpha (\beta)) \sem \{ ( u_\gamma)_\beta  \} & \\
  \hfill  \text{ or } i \neq g_\gamma (\beta) \} & \text{if } \beta \in F_\alpha \\
q'  (\beta)  & \text{otherwise.}
\end{cases}
\] 
Clearly $q_\alpha^{\gamma + 1} \leq_{F_\alpha ,\alpha} q_\alpha^\gamma$.
This completes the recursive construction. Let $p_{\alpha + 1} = \bigwedge_{\gamma < 2^{|\alpha|} } q_\alpha^\gamma$
and $\B_\alpha = \{ r_\alpha^\gamma : \gamma < 2^{|\alpha|} \}$. Clearly,  $p_{\alpha+1} \leq_{F_\alpha, \alpha} p_\alpha$ and $\B_{\alpha}$ 
has size at most $2^{|\alpha|}$ and is predense below $p_{\alpha+1}$. Finally define $F_{\alpha + 1}$ by adding
a single element to $F_\alpha$ and by guaranteeing via a book-keeping argument that the union of the $F_\alpha$
will agree with the union of the $\supp (p_\alpha)$.

For the limit step, suppose we already have constructed $(p_\alpha , F_\alpha: \alpha < \delta)$ where $\delta$ is a limit ordinal. We let $p_\delta = \bigwedge_{\alpha< \delta} p_\alpha$ and $F_\delta = \bigcup_{\alpha < \delta} F_\alpha$.

Take the fusion $q = \bigwedge_{\alpha < \kappa} p_\alpha$ of the sequence $(p_\alpha , F_\alpha : \alpha < \kappa)$. Then $q \leq p_\alpha$ 
for all $\alpha$ which implies that for all $\alpha$, $\B_\alpha$ is predense below $q$. 
Adjusting the construction for the finite ordinals, if necessary, we may 
assume that $| \B_\alpha | \leq 2^{ |\alpha|}$ for all $\alpha < \kappa$, as required.
\end{proof}

\begin{thm}   \label{Sacks-model}
Assume GCH and let $\lambda > \kappa^+$ be a cardinal with $\cf (\lambda) > \kappa$. 
Forcing with $\SS_\kappa^\lambda$ yields a generic extension in which  
$\dd_\kappa (\p\in^*) = \dd_h (\in^*) = \kappa^+$ and $\dd_\kappa (\in^*) = 2^\kappa = \lambda$.
\end{thm}

\begin{proof}
$\dd_h (\in^*) = \kappa^+$ follows immediately from Main Lemma~\ref{prod-Sacksprop}. By earlier results, this also implies
that all other cardinals we have considered are equal to $\kappa^+$. Furthermore, $2^\kappa = \lambda$ is an easy
consequence of GCH and $\cf (\lambda) > \kappa$, using Main Lemma~\ref{prod-Sacksprop} in the $\cf (\lambda) =
\kappa^+$ case. Hence we are left with showing $\dd_\kappa (\in^*) \geq \lambda$.
This follows almost, but -- since we are dealing with a product and not an iteration -- not quite, from Proposition~\ref{S-notSacksprop}.

Assume $\mu < \lambda$, and let $\{ \dot \varphi_\gamma : \gamma < \mu \}$ be $\SS_\kappa^\lambda$-names for
slaloms in $\Loc_\kappa$. By the $\kappa^{++}$-cc, we may find sets $A_\gamma \sub \lambda$, $\gamma < \mu$,
of size at most $\kappa^+$ such that $\dot \varphi_\gamma$ is added by the subforcing $\SS_\kappa^{A_\gamma}$.
Fix $\beta \in \lambda \sem \bigcup_{\gamma < \mu} A_\gamma$. Let $\dot s_\beta \in 2^\kappa$ be the name for the generic
$\kappa$-Sacks function added in coordinate $\beta$. Define the name $\dot f \in \kappa^\kappa$ by 
$\dot f (\alpha) =  g^{-1} ( \dot s_\beta \re \alpha)$ for $\alpha < \kappa$ where $g$ is again a bijection between $\kappa$ and $2^{< \kappa}$,
like in Proposition~\ref{S-notSacksprop}. We will prove that $\dot f$ is forced not to be localized by any $\dot\varphi_\gamma$.
This is clearly sufficient.

Fix $\gamma < \mu$. Also fix $p \in \SS_\kappa^\lambda$  and a cardinal $\alpha_0 < \kappa$. 
We need to find $\alpha \geq \alpha_0$ and $q \leq p$ such that $q \forces \dot f (\alpha ) \notin \dot \varphi_\gamma (\alpha)$.
Let $p^0 = p \re A_\gamma \in \SS_\kappa^{A_\gamma}$. Construct a decreasing chain $(p^0_\alpha : \alpha < \kappa )$ 
of conditions in $\SS_\kappa^{A_\gamma}$ with $p^0_0 = p^0$ and a 
slalom $\varphi \in \Loc_\kappa$ such that $p^0_\alpha \forces \dot \varphi_\gamma \re \alpha = \varphi \re \alpha$.

Next, as in the proof of Proposition~\ref{S-notSacksprop}, recursively construct a strictly increasing sequence of cardinals 
$(\alpha_n < \kappa : n \geq 1)$ such that $\alpha_1 > \alpha_0$ and $\Split_{\alpha_n} (p (\beta)) \sub 2^{\leq \alpha_{n+1}}$
for every $n \in \omega$.
Put $\alpha = \sup \{ \alpha_n : n \in \omega \}$.
Again we see that $|p(\beta) \cap 2^{\alpha}| = | \Split_\alpha (p(\beta)) | = 2^{|\alpha|} > \alpha$. 
Hence we can find $u \in p(\beta) \cap 2^{\alpha}$ such that
$g^{-1} (u) \notin \varphi (\alpha)$. Now define a condition $q$ by 
\begin{itemize}
\item $q(\beta) = (p(\beta))_u$,
\item $q \re A_\gamma = p^0_{\alpha + 1}$,
\item $q (\delta) = p (\delta)$ for $\delta \notin A_\gamma \cup \{ \beta \}$.
\end{itemize}
Clearly $q$ forces $\dot s_\beta \re \alpha = u$ and, thus, $\dot f (\alpha) = g^{-1} (u) \notin \varphi (\alpha) = \dot \varphi_\gamma (\alpha)$,
as required. 
\end{proof}

A detailed analysis of the two preceding proofs yields that for any function $h :\kappa \to \kappa$ with 
$\lim_{\alpha \to \kappa} h(\alpha) = \kappa$ that is above the power set function on a club (that is,
$\{ \alpha < \kappa : h(\alpha) \geq 2^{|\alpha|} \}$ contains a club), the $\kappa$-support product of
generalized Sacks forcing satisfies the $h$-Sacks property and, thus, $\dd_h (\in^*) = \kappa^+$ 
in the generic extension, while for any such function $h$ that is strictly below the power set function
on a stationary set, $\dd_h (\in^*)$ is large in the extension (see the comment after Proposition~\ref{S-notSacksprop}).
For example, for any continuous function $h:\kappa\to\kappa$, $\mathfrak{d}_h(\in^*)$ is large in the extension, as $h$ is the identity on a club.

Our results suggest some questions.

\begin{ques}
Is $\bb_\kappa (\in^*) < \bb_h (\in^*)$ consistent where $h$ is the power set function?
\end{ques}

This would be the consistency result dual to the one of Theorem~\ref{Sacks-model}. A natural approach
would be to iterate an appropriate version of the generalized localization forcing of Subsection~\ref{gen-loc}
with supports of size less than $\kappa$. However, we do not have the preservation results necessary to
show that $\bb_\kappa (\in^*)$ stays small. If this indeed could be done the dual model should provide an alternative
proof of Theorem~\ref{Sacks-model}.

\begin{ques}
Is it consistent that three cardinals of the form $\dd_h (\in^*)$ for different $h \in \kappa^\kappa$ 
simultaneously assume distinct values?
\end{ques}

A still simpler question might be:

\begin{ques}
Is it consistent that for some function $g$ strictly dominating the power set function $h$,
$\dd_g (\in^*) < \dd_h (\in^*)$ is consistent?
\end{ques}


\subsubsection{Generalized Silver forcing}

We finally note that instead of generalized Sacks forcing we may consider generalized Silver forcing
(see e.g.~\cite[Example 3.2]{FKK16} for this forcing notion) and add many $\kappa$-Silver functions
either with a $\kappa$-support product or a $\kappa$-support iteration. It is not difficult to see that 
this has the same effect on our cardinals, that is, $\dd_\kappa (\in^*)$ is increased to $2^\kappa$
and $\dd_h (\in^*)$ stays at $\kappa^+$ where $h$ is the power set function.


\subsection{Generalized Miller forcing}

\subsubsection{The single step}

Versions of generalized Miller forcing were investigated by Friedman, Honz\'ik, and Zdomskyy
in~\cite{FZ10} and~\cite{FHZ13}. Assume $\F$ is a normal and $\kappa$-complete non-principal filter
on $\kappa$. If $T \sub \kappa^{<\kappa}$ is a tree and $u \in T$, $u$ {\em $\F$-splits} (or simply {\em splits})
in $T$ if $\succ_T (u) = \{ \beta < \kappa : u \ha \beta \in T\}$ belongs to $\F$.

\begin{defi}
{\em Generalized Miller forcing $\MI_\kappa^\F$ with the filter $\F$} is defined as follows:
\begin{itemize}
\item conditions in $\MI_\kappa^\F$ are subtrees $T$ of $\kappa^{<\kappa}$ consisting of strictly increasing sequences such that:
\begin{enumerate} 
\item each $u \in T$ has a splitting extension $t \in T$, that is, $u \sub t$ and $t$ $\F$-splits in $T$ and, moreover, if $u$ does not
   $\F$-split, then $|\succ_T (u)| = 1$,
\item if $\alpha < \kappa$ is a limit ordinal, $u \in \kappa^\alpha$, and $u \re \beta \in T$ for every $\beta < \alpha$, then $u \in T$,
\item if $\alpha < \kappa$ is a limit ordinal, $u \in \kappa^\alpha$, and for arbitrarily large $\beta < \alpha$, $u \restriction \beta$ $\F$-splits in $T$,    
   then $u$ $\F$-splits in $T$;
\end{enumerate}
\item the order is given by inclusion, i.e. $T \leq S$ if $T \subseteq S$. 
\end{itemize}
\end{defi}

$\MI_\kappa^\F$ generically adds a function $m = m_G$ from $\kappa$ to $\kappa$, called a {\em $\kappa$-Miller function}
and given by $m = \bigcap \{T: T \in G \}$, where $G$ is a $\MI_\kappa^\F$-generic filter. A straightforward
genericity argument shows that $m$ is unbounded over the ground model functions.

Since the filter $\F$ is $\kappa$-complete, it is easy to see that $\MI_\kappa^\F$ is $< \kappa$-closed.
As for generalized Sacks forcing,
if $T \in \MI_\kappa^\F$ and $u \in T$, let $T_u = \{ t \in T : t \sub u$ or $u \sub t\}$.
For $T \in \MI_\kappa^\F$, $\stem(T)$ is the unique splitting node that is comparable with all elements in $T$. 
Also let $\Split (T)$ be the set of all $u \in T$ which split in $T$. Given $u \in \kappa^{< \kappa}$
let $\ell (u)$ denote the {\em length} of $u$, i.e. the unique $\alpha$ such that $u \in \kappa^\alpha$.

\begin{defi}[The $\alpha$-th splitting level of $T$, see~\cite{FZ10}]
Let $\Split_\alpha (T)$ be the set of all $u \in \Split (T)$ such that 
\begin{itemize}
\item $\{ v \in \Split (T) : v \subsetneq u \}$ has order type at most $\alpha$,
\item for all $v \subsetneq u$ in $\Split (T)$, $u (\ell (v) ) \cap \succ_T (v)$ has order type at most $\alpha$. 
\end{itemize}
\end{defi}

Thus $\Split_0 (T) = \{ \stem (T) \}$ and we note for later use that
$|\Split_\alpha (T) | \leq | (\alpha + 1)^{ \alpha + 1 } |$. Also note that $\Split_\alpha (T) \subset \Split_\beta (T)$
for $\alpha < \beta$ and that $\Split (T)$ is the increasing union of the $\Split_\alpha (T)$.
Using again the splitting levels $\Split_\alpha (T)$, define the
{\em fusion orderings} $\leq_\alpha$ for $\alpha < \kappa$ by $S \leq_\alpha T$ if $S \leq T$ and $\Split_\al(T)= \Split_\al(S)$.

\begin{defi}[Fusion sequence]
A sequence of conditions $(S_\al: \al< \ka) \subseteq \MI_\kappa^\F$ is a {\em fusion sequence} if 
\begin{itemize}
\item $S_{\alpha +1} \leq_\al S_\al$ for every $\alpha < \kappa$, 
\item $S_\delta= \bigcap_{\al< \delta} S_\al$  for limit $\delta< \ka$. 
\end{itemize}
\end{defi}

The Fusion Lemma~\cite[Lemma 2.3]{FZ10} says that
for a fusion sequence $(S_\al: \al< \ka)$, the {\em fusion} $S = \bigcap_{\alpha < \kappa} S_\alpha$ is a condition in $\MI_\kappa^\F$. 
The proof uses $\kappa$-completeness and normality of the filter $\F$.

For $\kappa = \omega$, classical Miller forcing does not add Cohen reals. However,
for inaccessible $\kappa$, forcings of the type $\MI_\kappa^\F$ may add $\kappa$-Cohen
functions.

\begin{prop}
$\kappa$-Miller forcing $\MI_\kappa^\C$ with the club filter $\mathcal{C}$ adds a $\kappa$-Cohen function.
\end{prop}

\begin{proof}
In the ground model $V$, let $(S_\alpha: \alpha< \kappa)$ be a partition of $\kappa$ into stationary sets. Define $f \in \kappa^\kappa$ by letting $f(\beta)$ be the unique $\alpha$ such that $\beta \in S_\alpha$. Also, let  $\psi$ be a bijection between $\ka$ and $2^{<\ka}$. Finally, let $\dot m$ be the $\kappa$-Miller function added by $\mathbb{MI}_\kappa^\mathcal{C}$.
 
In the generic extension, define the function $g^* : \kappa \to 2^{<\kappa}$ as the composition $g^\ast= \psi \circ f \circ m$.  We claim that the function $g: \ka \to 2$ that concatenates the values of $g^\ast$, that is,
\[ g= g^\ast(0) \ha g^\ast(1) \ha \ldots \ha g^\ast(\al) \ha ... \] 
is $\kappa$-Cohen generic over the ground model $V$. To show this, take $T \in \mathbb{MI}_\kappa^\mathcal{C}$ and $\mathcal{D} \subseteq \mathbb{C}_{\kappa}$ open dense. It is enough to prove that we can find $u \in \mathcal{D}$ and $S \leq T$ such that $S \Vdash u \subseteq \dot{g}$.

Put $\sigma = \stem(T)$, consider the composition $v^\ast = \psi \circ f \circ \sigma \in (2^{<\kappa})^{<\kappa}$, and let $v$ be its concatenation. Then $v \in \mathbb{C}_\ka$ and so there exists $u \in \D$ such that $u \supseteq v$. Note that $T \Vdash \dot m \restriction {\ell( \sigma )} = 
\sigma$, and hence $T \forces \dot g \re \ell (v) = v$.
Define a sequence $w \in 2^{<\kappa}$ by $w (\alpha) = u (\ell (v) + \alpha)$ for all $\alpha$ with $\ell (v) + \alpha < \ell (u)$. 
(That is, $u = v \ha w$ is the concatenation of $v$ and $w$.) Since $\succ_T (\sigma)$ is a club set, there is $\beta \in \succ_T (\sigma) \cap S_{\psi^{-1} (w)}$. Let $\tau = \sigma \ha \beta$ and $S = T_\tau$. Then $\ell (\tau) = \ell (\sigma) + 1$, $\tau ( \ell (\sigma)) = \beta$, 
$\psi (f ( \tau  ( \ell (\sigma)) )) = w$, and letting $u^* = \psi\circ f \circ \tau \in (2^{<\kappa})^{<\kappa}$, we see that its concatenation is exactly $u = v \ha w$. Hence $S \forces u \sub \dot g$, as required.
\end{proof}


For the subsequent discussion, assume that $\U$ is a $\kappa$-complete normal ultrafilter on a measurable cardinal $\kappa$.

\begin{prop}[Pure decision property]   \label{pure-dec}
Let $\varphi$ be a sentence of the forcing language. Assume $T \in \mathbb{MI}_\kappa^\mathcal{U}$. Then there is $S \leq T$ with the same stem such that $S$ decides $\varphi$, that is, $S \Vdash \varphi$ or $S \Vdash \neg \varphi$.
\end{prop}

\begin{proof}
Letting $T$ and $\varphi$ be as in the statement of the proposition and putting $\sigma= \stem(T)$, there is a set $X \in \mathcal{U}$ such that $\sigma \ha i \in T$ for all $i \in X$. Now, given $i \in X$ find $S_i \leq T_{\sigma\ha i}$ such that $S_i$ decides $\varphi$.
Since $\mathcal{U}$ is an ultrafilter, one of the sets $X_0= \{i \in X : S_i \Vdash \varphi \}$ or $X_1= \{i \in X : S_i \Vdash \neg \varphi \}$ must belong to $\U$. Let
$S = \bigcup_{i \in X^\ast} S_i$, where $X^\ast$ is either $X_0$ or $X_1$, depending on which one of them belongs to $\mathcal{U}$. Clearly, $S\leq T$ and $S$ decides $\varphi$ (specifically if $X_0 \in \mathcal{U}$, we have $S \Vdash \varphi$ and otherwise $S \Vdash \neg \varphi$).
\end{proof}

\begin{defi}[Generalized Laver property]
Let $h \in \kappa^\kappa$ with $\lim_{\alpha \to \kappa} h(\alpha) = \kappa$. 
A forcing notion $\mathbb{P}$ has the {\em generalized $h$-Laver property} if for every condition $p \in \mathbb{P}$, every $g \in \kappa^\kappa$, and every $\mathbb{P}$-name $\dot{f}$ such that $p \forces \dot f (\alpha) < g(\alpha)$ for all $\alpha < \kappa$, there are a condition $q \leq p$ and an $h$-slalom $\varphi \in \Loc_h$ such that $q \Vdash \dot{f}(\alpha) \in \varphi (\alpha)$ for all $\alpha < \kappa$. If $h = \id$, we say $\PP$ has the {\em generalized Laver property}.
\end{defi}

The generalized Laver property is closely related to the generalized Sacks property: Say that a forcing notion $\PP$ is {\em $\kappa^\kappa$-bounding} if for all $\PP$-names $\dot f \in \kappa^\kappa$ and all $p \in \PP$, there are $g \in \kappa^\kappa$ and $q \leq p$ such that $q \forces \dot f (\alpha) < g (\alpha)$ for all $\alpha < \kappa$. Then $\PP$ has the generalized $h$-Sacks property if and only if $\PP$ is $\kappa^\kappa$-bounding and has the generalized $h$-Laver property.

In analogy to the countable case we have:

\begin{prop}
Let $h \in \kappa^\kappa$ with $\lim_{\alpha \to \kappa} h(\alpha) = \kappa$.
If $\mathbb{P}$ has the generalized $h$-Laver property, then $\mathbb{P}$ does not add $\kappa$-Cohen functions. 
\end{prop}

\begin{proof} 
Let $\dot{f}$ be a $\mathbb{P}$-name for a function in $2^\kappa$, and consider a partition of $\kappa$ into $\kappa$ many intervals $(I_\alpha: \alpha< \kappa)$ such that $ \lvert I_\alpha \lvert =  |h(\alpha)|$. Now define a name for a function $\dot g \in (\Fn (\kappa,2,\kappa))^\kappa$ by $\dot{g}(\alpha)= \dot{f} \restriction I_\alpha$. 

Note that $\dot{g}$ does not belong to the ground model $V$, but it is possible to find a function $G \in V$ that is forced to bound it.
Indeed, since the possible values of $\dot{g}(\alpha)$ are bounded by the number of functions from the interval $I_\alpha$ to $2$, that is, by $2^{\lvert I_\alpha \lvert}= 2^{|h(\alpha) |} < \ka$ (where the latter holds by the strong inaccessibility of $\ka$), we can define $G \in ((\Fn (\kappa,2,\kappa))^{<\kappa})^\kappa$ by $G(\alpha)= 2^{I_\alpha}$. 

Thus we can use the generalized Laver property: given a condition $p \in \mathbb{P}$ we can find a condition $q \leq p$ and an $h$-slalom $\varphi : \kappa \to (\Fn (\kappa,2,\kappa))^{<\kappa}$ (that is, $\lvert \varphi (\alpha)\lvert \leq |h(\alpha)| < 2^{|h(\alpha)|} = | G (\alpha )|$ for all $\alpha$) such $q \Vdash \dot{g}(\alpha) \in \varphi (\alpha)$ for all $\alpha < \kappa$. Thus, if we define the set $A$ by
\[ A = \{ x \in 2^\kappa : \forall \alpha <\kappa \; (x \restriction I_\alpha \in \varphi (\alpha))\}= \bigcap_{\alpha < \kappa} \{ x \in 2^\kappa : x \restriction I_\alpha \in \varphi (\alpha)\},\] 
we see that $A$ is nowhere dense and that $q \forces \dot{f} \in A$. This implies that $\dot{f}$ is forced not to be $\kappa$-Cohen.
\end{proof}

\begin{prop} 
Let $h$ be the power set function.
If $\U$ is a $\kappa$-complete normal ultrafilter on $\ka$ then Miller forcing $\mathbb{MI}_\kappa^{\mathcal{U}}$ with $\U$ has the generalized $h$-Laver property. 
\end{prop}

\begin{proof}
Let $T \in \mathbb{MI}_\kappa^{\mathcal{U}}$ and $g \in \kappa^\kappa $, and let $\dot{f}$ be a $\mathbb{MI}_\kappa^{\mathcal{U}}$-name for an element in $\kappa^\kappa$ such that $T \Vdash \dot{f} (\alpha) < g (\alpha)$ for all $\alpha < \kappa$. 

Recursively we construct a fusion sequence $(S_\alpha : \alpha < \kappa)$ and sets $B_\alpha \sub \kappa$ such that:
\begin{itemize}
\item $S_{\alpha +1} \Vdash \dot{f}(\alpha) \in B_\alpha$,
\item $\lvert B_\alpha \lvert\leq 2^{\lvert \alpha \lvert}$ for infinite $\alpha$.
\end{itemize}
In the basic step, let $S_0 = T$. 

In the successor step, suppose we have already $S_\alpha$. Fix $t \in \Split_\alpha (S_\alpha)$.
Using the pure decision property (Proposition~\ref{pure-dec}) together with the $\kappa$-completeness of the ultrafilter
and the fact that the possible values of $\dot f (\alpha)$ are bounded by $g(\alpha)$, we see that there are $\gamma_t < g (\alpha)$
$D_t \sub \succ_{S_\alpha} (t)$ belonging to $\U$ and trees $(S_{\alpha + 1})_{t \ha \beta} \leq (S_\alpha)_{t \ha \beta}$ 
for $\beta \in D_t$ such that $(S_{\alpha + 1})_{t \ha \beta} \forces \dot f (\alpha) = \gamma_t$.

(More explicitly, for each $\beta \in  \succ_{S_\alpha} (t)$, choose $\gamma_t^\beta < g (\alpha)$
and $(S_{\alpha + 1})_{t \ha \beta} \leq (S_\alpha)_{t \ha \beta}$ such that 
$(S_{\alpha + 1})_{t \ha \beta} \forces \dot f (\alpha) = \gamma_t^\beta$. Then find $\gamma_t$
and $D_t \in \U$ such that $\gamma_t^\beta = \gamma_t$ for all $\beta \in D_t$.
This is possible because $|\{ \gamma_t^\beta : \beta \in  \succ_{S_\alpha} (t) \}| \leq | g(\alpha) | < \kappa$.)

At the end let $S_{\alpha + 1} = \bigcup_{t \in \Split_\alpha (S_\alpha) } \bigcup_{\beta \in D_t} (S_{\alpha + 1})_{t \ha \beta}$
and $B_\alpha = \{ \gamma_t : t \in \Split_\alpha (S_\alpha) \}$. Since $| \Split_\alpha (S_\alpha) | \leq |(\alpha + 1) ^{
\alpha + 1} |$, $|B_\alpha | \leq 2^{|\alpha|}$ for infinite $\alpha$ follows. $S_{\alpha +1} \Vdash \dot{f}(\alpha) \in B_\alpha$
is obvious by construction.

In the limit step, let $S_\alpha= \bigcap_{\beta < \alpha} S_\beta$. 

Finally, take the fusion $S$ of the sequence $(S_\alpha : \alpha < \kappa)$. 
Then $S \leq_{\alpha} S_\alpha$ for all $\alpha$ which implies that for all $\al < \ka$, $S \Vdash \dot{f}(\al) \in B_\al$. Hence, if we 
adjust the construction for finite ordinals and define $\varphi (\al)= B_\al$, then $\lvert \varphi (\al)\lvert \leq 2^{\lvert \al \lvert}$ and $S \Vdash \dot{f}(\al) \in \varphi(\al)$ for all $\alpha < \kappa$, and we have the generalized $h$-Laver property.
\end{proof}

Our reason for investigating the generalized Laver property for forcings of type $\MM_\kappa^\U$ was the hope
to obtain an alternative proof for the consistency of $\cov (\M_\kappa) < \dd_\kappa$, originally obtained
by Shelah~\cite{Shta} (see~\cite[7.3.E]{BJ95} for $\cov (\M) < \dd$ in the Miller model). 
For this, however, one would need preservation of the generalized Laver property in
iterations with support of size $\kappa$.

\begin{ques}
Is the generalized Laver property preserved under $\kappa$-support iterations?
\end{ques}

Since there are no preservation theorems for such iterations, this seems to be out of reach. More specifically
we may ask:

\begin{ques}   \label{Miller-Laver}
Let $h$ be the power set function. Assume $\kappa$ is an indestructible supercompact cardinal.
Does the $\kappa$-support iteration of forcings of type $\MI_\kappa^\U$ have the generalized
$h$-Laver property?
\end{ques}

The reason for considering indestructible supercompact $\kappa$ here is that we need that the
intermediate extensions still contain $\kappa$-complete normal ultrafilters on $\kappa$. 

Showing directly that the whole iteration has the generalized $h$-Laver property by a generalized 
fusion argument seems more feasible. This can indeed be done for generalized Sacks forcing but uses
that we can decide the splitting levels -- something which is not possible for generalized Miller
forcing where we would have to work with names. It also can be done for any forcing with reasonable
fusion properties in case $\kappa = \omega$ -- but this hinges on the fact that the sets $F_\alpha$
arising in the generalized fusion in this case are finite. In fact, this was the original approach taken
in work of Laver~\cite{La76} and Baumgartner~\cite{Ba78} before Shelah~\cite{Sh98} developed the theory of
preservation theorems.

It is known (see~\cite[Theorem 2.9]{FZ10} and~\cite[Theorem 2.9]{FHZ13}) that generalized fusion
for the $\kappa$-support iteration of generalized Miller forcing shows the preservation of $\kappa^+$
(and if $2^\kappa = \kappa^+$ and the length of the iteration is at most $\kappa^{++}$, then the $\kappa^{++}$-cc still holds so that
all cardinals are preserved). This type of fusion argument however does not help for the generalized 
Laver property. 

For $\kappa = \omega$, the Laver property of Laver or Mathias forcing is used to show the consistency of 
$\add (\M) < \bb$~\cite[7.3.D and 7.4.A]{BJ95}. Since $\add (\M_\kappa) < \bb_\kappa$ implies $\cov (\M_\kappa) < \dd_\kappa$ 
by Corollary~\ref{Truss-cor} (for any $\kappa$), the former consistency seems to be even harder for
strongly inaccessible $\kappa$. In fact, Laguzzi~\cite{Lata} proved that any reasonable generalization of
Mathias forcing to $\kappa$ adds $\kappa$-Cohen functions. Also, for generalized Laver forcing
it is unclear whether any form of the generalized Laver property can hold, and in the case that 
$\kappa$ is measurable and $\U$ is a normal $\kappa$-complete ultrafilter on $\kappa$, it is shown in~\cite[Lemma 18]{BFFM16}
that the Mathias and Laver forcings are equivalent. Thus the following problem
seems to be out of reach.

\begin{ques}   \label{add-b-ques}
Is it consistent for strongly inaccessible (or even supercompact) $\kappa$ that $\add (\M_\kappa) < \bb_\kappa$?
That $\dd_\kappa < \cof (\M_\kappa)$?
\end{ques}


\subsubsection{The product}

We close our work with some results on products of generalized Miller forcing. 
Since the classical Miller forcing $\MI$ has the Laver property, it follows that it does not add Cohen reals. Spinas~\cite{Sp01}  proved  that the product of two copies of Miller forcing $\MI^2$ does not add Cohen reals. However by a result of Veli\v ckovi\'c and Woodin~\cite{VW98}, the product of three copies of Miller forcing $\MI^3$ adds a Cohen real. For strongly inaccessible $\kappa$, two $\kappa$-Miller functions are sufficient to
get a $\kappa$-Cohen function. Roughly speaking, the reason for this is that given trees $T, S \in \MI_\kappa^\F$, the set of places where
both $T$ and $S$ split contains a club.

\begin{thm}   \label{prod-Miller-Cohen}
Let $\F$ be a $\ka$-complete normal filter on $\kappa$. Then the product $\QQ^\ast = \mathbb{MI}_\kappa^{\mathcal{F}} \times \mathbb{MI}_\kappa^{\mathcal{F}}$ adds a $\kappa$-Cohen function. 
\end{thm}

\begin{proof}
Let $m_0$ and $m_1$ be the generically added $\kappa$-Miller functions. We say that $\al < \ka$ is an {\em oscillation point} of $m_0$ and $m_1$ if there is $\ga< \al$ such that 
\begin{itemize}
\item {\em either} $m_0(\al) > m_1(\al)$ and $m_0(\be)< m_1(\be)$ for all $\beta$ with $\ga \leq \be < \al$ 
\item {\em or} $m_1(\al) > m_0(\al)$ and $m_1(\be)< m_0(\be)$ for all $\beta$ with $\ga \leq \be < \al$. 
\end{itemize}
Let $A$ denote  the set of oscillation points of $m_0$ and $m_1$ and $C$, the set of limit points of $A$. 
It is easy to see that $A$ is unbounded in $\kappa$ and, thus, $C$ is club in $\kappa$.

(Indeed, let $\alpha_0 < \kappa$ and $(S,T) \in \QQ^*$. Put $\sigma = \stem (S)$ and $\tau = \stem (T)$.
Without loss of generality we may assume $\alpha_0 < \ell (\sigma) < \ell (\tau)$. First extend $S$ to $S'$ such that
the stem $\sigma '$ of $S'$ is longer than $\ell (\tau)$ and $\sigma ' (\be) > \tau (\be)$ for all
$\be$ with $\ell (\sigma ) \leq \be < \ell (\tau)$. Next extend $T$ to $T'$ such that
the stem $\tau '$ of $T'$ properly extends $\tau$ and $\tau ' (\ell (\tau)) > \sigma ' (\ell (\tau))$.
Then $(S',T')$ forces that $\ell (\tau)$ belongs to $\dot A$.)

Now, let $\{\gamma_\al : \al < \ka\}$ be the continuous increasing enumeration of $C$ and define $c: \ka \to 2$ by:
$$c(\al) = \left\{
\begin{array}{c l}
 0 & \text{ if } m_0(\gamma_\al) \leq m_1(\gamma_\al)\\
 1 & \text{ if } m_0(\gamma_\al) > m_1(\gamma_\al).
\end{array}
\right.
$$
We argue that $c$ is $\ka$-Cohen generic. Let $(S,T) \in \QQ^*$ and let $D \subseteq \CC_\ka$ be dense. We shall find $(S', T') \leq (S,T)$ and $w \in D$ such that $(S', T') \Vdash w \subseteq \dot{c}$. Let $\sigma= \stem(S)$ and $\tau=\stem(T)$.
By an easy pruning argument we see that we may assume without loss of generality that the stems have the same length and
let $\delta = \ell(\sigma)= \ell(\tau)$. Then $(S,T)\Vdash \dot{m}_0 \restriction \delta = \sigma$ and $\dot{m}_1 \restriction \delta = \tau$.
We may also suppose that $\delta$ is a limit of oscillation points of $\sigma$ and $\tau$. This means that there are $\epsilon < \ka$
and a continuous increasing sequence $(\gamma_\alpha : \alpha \leq \epsilon)$ such that
\[ (S,T) \forces `` \dot C \cap \delta = \{ \dot\ga_\al : \al < \epsilon \} = \{ \ga_\al : \al < \epsilon \} \text{ and } \dot \ga_\epsilon = \gamma_\epsilon=
\delta \in \dot C". \]
Now define the  function $v \in 2^{<\kappa}$ with domain $\dom(v) = \epsilon$ and
$$v(\al) = \left\{
\begin{array}{c l}
 0 & \text{ if } \sigma(\gamma_\al) \leq \tau(\gamma_\al)\\
 1 & \text{ otherwise }
\end{array}
\right.
$$
for $\alpha < \epsilon$. Clearly $(S,T) \forces v \sub \dot c$ and no further values of $\dot c$ are decided by $(S,T)$.

Using the density of $D$ take $w \in D$ such that $w \supseteq v$. Let $\theta$ be the order type of
$\dom (w)  - \dom (v)$, that is, $\epsilon + \theta = \dom (w)$. Let $\{ \xi_\eta : \eta < \theta \}$
enumerate the limit ordinals in $\omega \times \theta$, that is, $\xi_0 = 0, \xi_1 = \omega, ...$.
Now recursively construct $\sigma_\xi$ and $\tau_\xi$, $\xi \in \omega \times \theta$,
such that
\begin{itemize}
\item $\sigma_0 = \sigma$ and $\tau_0 = \tau$,
\item the $\sigma_\xi$ form a strictly increasing sequence of splitting nodes of $S$ and the $\tau_\xi$ form a 
   strictly increasing sequence of splitting nodes of $T$,
\item if $\xi$ is a limit ordinal, then $\sigma_\xi = \bigcup_{\zeta < \xi} \sigma_\zeta$, $\tau_\xi = \bigcup_{\zeta < \xi}
   \tau_\zeta$, and $\ell (\sigma_\xi) = \ell (\tau_\xi)$,
\item if $\xi$ is a successor ordinal, then $\ell (\sigma_\xi) < \ell (\tau_\xi) < \ell (\sigma_{\xi + 1})$,
   for $\be$ with $\ell (\sigma_\xi) \leq \be < \ell (\tau_\xi)$, $\tau_\xi (\be) < \sigma_{\xi + 1} (\be)$,
   and for $\be$ with $\ell (\tau_\xi) \leq \be < \ell (\sigma_{\xi+1})$, $\sigma_{\xi+1} (\be) < \tau_{\xi + 1} (\be)$,
\item if $\xi_\eta$ is $0$ or a limit ordinal, then: 
\begin{itemize}
   \item if $w ( \epsilon + \eta) = 0$, $\sigma_{\xi_\eta +1} ( \be ) 
   < \tau_{\xi_\eta +1} ( \be)$ for all $\be$ with $\ell (\sigma_{\xi_\eta}) = \ell (\tau_{\xi_\eta}) \leq \be < \ell (\sigma_{\xi_\eta + 1})$,
   \item if $w ( \epsilon + \eta) = 1$, there is $\delta$ with $\ell (\sigma_{\xi_\eta}) = \ell (\tau_{\xi_\eta}) < \delta <
   \ell (\sigma_{\xi_\eta + 1})$ such that $\tau_{\xi_\eta +1} ( \be ) 
   < \sigma_{\xi_\eta +1} ( \be)$ for all $\be$ with $\ell (\sigma_{\xi_\eta}) = \ell (\tau_{\xi_\eta}) \leq \be < \delta$ and
   $\sigma_{\xi_\eta +1} ( \be ) < \tau_{\xi_\eta +1} ( \be)$ for all $\be$ with $\be$ with $\delta \leq \be < \ell (\sigma_{\xi_\eta + 1})$.
\end{itemize}
\end{itemize}
By alternately extending the splitting nodes in $S$ and in $T$, respectively, it is easy to see that
this construction can be carried out.

Now let $\sigma ' = \bigcup_{\xi \in\omega\times\theta} \sigma_\xi$ and $\tau' = \bigcup_{\xi \in\omega\times\theta} \tau_\xi$.
Clearly $\ell (\sigma') = \ell (\tau ')$.
Let $S' = S_{\sigma '}$ and $T' = T_{\tau '}$. Note that, by construction, the only new ordinals which are limits
of oscillation points of $\sigma '$ and $\tau'$ are the $\ell (\sigma_{\xi_\eta}) = \ell (\tau_{\xi_\eta})$, $\eta < \theta$.
This means that
\[  (S' , T' ) \forces `` \dot C \cap \ell (\sigma ') = \{ \dot\ga_\al : \al < \epsilon + \theta \} = \{ \ga_\al : \al < \epsilon \}
\cup \{ \ell (\sigma_{\xi_\eta}) : \eta < \theta \} ". \]
Furthermore, by the last item of the construction, we see that $(S',T') \forces w \sub \dot c$.
\end{proof}

It is well-known that the full product of countably many Cohen reals collapses the continuum to $\omega$ and thus,
by~\cite{VW98}, the same is true for the full product of countably many Miller reals. For strongly inaccessible $\kappa$, the
situation is different, the $\kappa$-support product of $\kappa$-Cohen forcing preserves $\kappa^+$~\cite[Proposition 24]{Frta} and therefore,
under GCH, all cardinals. We shall see below (Proposition~\ref{prod-Miller-pres}) the same is true for
$\kappa$-Miller forcing so that we may actually consider the product.

Let  $\calF$ be a $\ka$-complete normal filter on $\ka$.

\begin{defi}
For a  set $A$ of ordinals, $\mathbb{MI}^\calF_{\ka, A}$ is the {\em $\ka$-support product} of $\mathbb{MI}^\calF_\ka$ 
with index set $A$, that is, $\mathbb{MI}^\calF_{\ka, A}$ consists of all functions $p: A \to \mathbb{MI}^\calF_\ka$ such that 
$\supp(p)=\{\be \in A: p(\be) \neq \ka^{<\ka}\}$
has size at most $\ka$. $\mathbb{MI}^\calF_{\ka, A}$ is ordered coordinatewise: $q \leq p$ if $q(\be) \leq p(\be)$ for all $\be \in A$.
\end{defi}

As in the $\ka$-Sacks case, given $\be \in A$, $\mathbb{MI}^\calF_{\ka, A}$ adds a $\mathbb{MI}^\calF_\ka$-generic function $m_\beta$ 
over the ground model. This forcing notion is $<\ka$-closed and, assuming $2^\kappa = \kappa^+$, has the $\ka^{++}$-cc. 
For every $F \subseteq A$ of size $<\ka$ and $\alpha < \kappa$, we define the 
\emph{fusion ordering} $\leq_{F,\alpha}$ as follows: $q \leq_{F,\alpha} p$ if  $q \leq p$ and for every $\be \in F$, $q(\be) \leq_\al p(\be)$.

\begin{defi}[Generalized Miller fusion]
$(p_\alpha, F_\al : \alpha < \kappa)$ is a {\em generalized fusion sequence} if $p_\al \in \mathbb{MI}^\calF_{\ka, A}$, $F_\al \in [A]^{<\ka}$, and
\begin{enumerate}
 \item $p_{\alpha +1} \leq_{F_\alpha, \alpha} p_\alpha$ and $p_\delta = \bigwedge_{\alpha < \delta} p_\alpha$ when $\delta$ is a limit ordinal $< \kappa$,
 \item $F_\alpha \subseteq F_{\alpha+1}$, $F_\delta = \bigcup_{\alpha< \delta} F_\alpha$ for limit $\delta< \kappa$ and $\bigcup_{\alpha< \kappa} F_\alpha = $ \linebreak $\bigcup_{\alpha< \kappa} \supp (p_\alpha)$.
\end{enumerate}
\end{defi}

By the analogue of the Generalized Fusion Lemma from Kanamori~\cite{Ka80} we see that given such a generalized fusion sequence 
$(p_\alpha, F_\al : \alpha < \kappa)$ the {\em fusion} $p = \bigwedge_{\alpha < \kappa} p_\alpha$ is a condition in $\mathbb{MI}^\calF_{\ka, A}$. 
This allows us to ensure the preservation of $\ka^+$. 

\begin{prop}  \label{prod-Miller-pres}
$\MI_{\kappa,A}^\F$ preserves $\kappa^+$.
\end{prop}

\begin{proof}
We first fix some notation. Let $\alpha < \kappa$, $T \in \MI^\F_\kappa$ and $u \in \Split_\alpha (T)$. 
Define $T_u^\alpha$ as follows: if $u$ is a final element of $\Split_\alpha (T)$, i.e., $\{ v \in \Split (T) : v \subsetneq u \}$ has order type
exactly $\alpha$, then $T_u^\alpha = T_u$; otherwise $T_u^\alpha = \{ t \in T_u : t \sub u$ or $t(\ell (u)) \cap \succ_T (u)$
has order type $> \alpha \}$. Next, for $q \in \MI^\F_{\kappa,A}$, $F \sub A$, and $\bar u = (u_\beta : \beta \in F)$, define $q_{\bar u}^\alpha$
by letting
\[ q_{\bar u }^\alpha (\beta)=
\begin{cases}
(q (\beta))_{ u_\beta }^\alpha & \text{if } \beta \in F \\
q  (\beta)  & \text{otherwise.}
\end{cases}
\] 

Now assume $p \in \MI^\F_{\kappa,A}$ and $\dot f $ is an $\MI^\F_{\kappa,A}$-name for a function from $\kappa$
to the ordinals. We recursively construct a generalized fusion sequence $(p_\alpha, F_\alpha : \alpha < \kappa )$ and sets
$B_\alpha$ of ordinals such that 
\begin{itemize}
\item $|F_\alpha| \leq |\alpha|$,
\item $|B_\alpha| \leq \max ( 2^{|\alpha|}, \omega)$,
\item for all $\zeta < \alpha$ and all sequences $\bar u = (u_\beta : \beta \in F_\alpha)$ with all $u_\beta \in \Split_\alpha (p_\alpha (\beta))$,
   if there is $q \leq_{F_\alpha,0} (p_{\alpha + 1})_{\bar u}^\alpha$ forcing a value to $\dot f (\zeta)$, then for some $\xi \in B_\alpha$,
   $(p_{\alpha + 1})^\alpha_{\bar u} \forces \dot f (\zeta) = \xi$.
\end{itemize}
For the basic step, let $p_0 = p$ and $B_0 = \emptyset$.

For the successor step, suppose that $p_\alpha$ and $F_\alpha$ have already been constructed. Let $\bar U = \{ \bar u =
(u_\beta : \beta \in F_\alpha) : u_\beta \in \Split_\alpha (p_\alpha (\beta))$ for all $\beta \in F_\alpha \}$. Since 
$|\Split_\alpha (p_\alpha (\beta))| \leq | (\alpha + 1)^{\alpha + 1} |$ for all $\beta$ and $|F_\alpha| \leq |\alpha|$, we see that
$|\bar U| = 2^{|\alpha|}$ for infinite $\alpha$. Let $((\bar u_\gamma , \zeta_\gamma) : \gamma < \lambda_\alpha )$
enumerate all pairs $(\bar u, \zeta) \in \bar U \times \alpha$ where $\lambda_\alpha = 
| \bar U \times \alpha | \leq \max (2^{|\alpha|},\omega)$. Recursively construct a decreasing chain $(q^\gamma_\alpha :
\gamma < \lambda_\alpha)$ of conditions and ordinals $(\xi_\alpha^\gamma : \gamma < \lambda_\alpha )$ such that
\begin{itemize}
\item $q_\alpha^0 = p_\alpha$, $q_\alpha^\delta \leq_{F_\alpha ,\alpha} q_\alpha^\gamma$ for all $\gamma \leq \delta$,
\item if there is $q \leq_{F_\alpha, 0} ( q_\alpha^\gamma)_{\bar u_\gamma}^\alpha$ forcing a value to $\dot f (\zeta_\gamma)$, then
   $(q_\alpha^{\gamma + 1})_{\bar u_\gamma }^\alpha \forces \dot f (\zeta_\gamma) = \xi^\gamma_\alpha$,
\item $q_\alpha^\delta = \bigwedge_{\gamma <\delta} q_\alpha^\gamma$ for limit ordinals $\delta$.
\end{itemize}
Since the basic step and the limit step are straightforward, it suffices to do the successor step of this recursion. 
Assume $q_\alpha^\gamma$ has been produced. If no $q$ as in the second clause exists, let $\xi_\alpha^\gamma = 0$ and
$q_\alpha^{\gamma + 1} = q_\alpha^\gamma$. If such a $q$ exists, say $q \forces \dot f (\zeta_\gamma) = \xi$, then
let $\xi_\alpha^\gamma = \xi$ and 
\[ q_\alpha^{\gamma + 1} (\beta) =
\begin{cases}
q (\beta) \cup \bigcup \{ (q_\alpha^\gamma (\beta) )_{u} : u \in \Split_\alpha (p_\alpha (\beta)) \sem \{ ( u_\gamma)_\beta  \} \} 
     & \text{if } \beta \in F_\alpha \\
q (\beta)  & \text{otherwise.}
\end{cases}
\] 
Clearly $q(\beta) = (q_\alpha^{\gamma + 1} (\beta) )^\alpha_{( u_\gamma)_\beta}$ for $\beta \in F_\alpha$ and
$q_\alpha^{\gamma + 1} \leq_{F_\alpha ,\alpha} q_\alpha^\gamma$.
This completes the recursive construction. Let $p_{\alpha + 1} = \bigwedge_{\gamma < \lambda_\alpha } q_\alpha^\gamma$
and $B_\alpha = \{ \xi_\alpha^\gamma : \gamma < \lambda_\alpha \}$. Clearly,  $p_{\alpha+1} \leq_{F_\alpha, \alpha} p_\alpha$ and $B_{\alpha}$ 
has size at most $\lambda_\alpha$. Finally define $F_{\alpha + 1}$ by adding
a single element to $F_\alpha$ and by guaranteeing via a book-keeping argument that the union of the $F_\alpha$
will agree with the union of the $\supp (p_\alpha)$.

For the limit step, suppose we already have constructed $(p_\alpha , F_\alpha: \alpha < \delta)$ where $\delta$ is a limit ordinal. We let $p_\delta = \bigwedge_{\alpha< \delta} p_\alpha$ and $F_\delta = \bigcup_{\alpha < \delta} F_\alpha$.

Take the fusion $q = \bigwedge_{\alpha < \kappa} p_\alpha$ of the sequence $(p_\alpha , F_\alpha : \alpha < \kappa)$. Then $q \leq p_\alpha$ 
for all $\alpha$. We claim that $q$ forces $\ran (\dot f ) \sub B$ where $B = \bigcup_{\alpha < \kappa} B_\alpha$.

To see this, let $r \leq q$ and assume $r \forces \dot f (\zeta) = \eta$ for some $\zeta < \kappa$ and some $\eta$. Choose $\alpha > \zeta$ such that
$u_\beta : = \stem (r (\beta)) \in \Split_\alpha (q (\beta)) = \Split_\alpha (p_\alpha (\beta))$ for all $\beta \in F_\alpha$. (This is clearly
possible because the sequence $(F_\alpha : \alpha < \kappa)$ is continuous.) Let $\bar u = (u_\beta : \beta \in F_\alpha )$. By
construction we must have $(p_{\alpha + 1})^\alpha_{\bar u} \forces \dot f (\zeta) = \xi$ for some $\xi \in B_\alpha$.
Since $r \leq_{F_\alpha , 0} q_{\bar u} \leq_{F_\alpha, 0} (p_{\alpha + 1}) _{\bar u}$, $r$ and $(p_{\alpha + 1} )^\alpha_{\bar u}$
are compatible and therefore $\eta = \xi \in B_\alpha \sub B$, as required.
\end{proof}

If we consider the product when $A=\lambda$ we obtain a model in which the cardinal invariants assume the same
values as in the $\kappa$-Cohen extension (see Proposition~\ref{Cohen-model}).

\begin{prop}
Assume $2^\kappa = \kappa^+$ and let $\lambda> \ka^+$ be a cardinal with $\lambda^\ka = \lambda$. Then, in the $\mathbb{MI}^\calF_{\ka, \lambda}$-generic extension,  $\non(\M_\ka)=\ka^+$ and $\cov(\M_\ka)=2^\ka=\lambda$ holds.
\end{prop}

\begin{proof}
The equality $2^\ka= \lambda$ follows from $\lambda^\ka = \lambda$. Using Proposition~\ref{prod-Miller-Cohen} 
we know that the product $\mathbb{MI}^\calF_{\ka, \lambda}$ adds $\lambda$-many $\ka$-Cohen functions.

For $\gamma < \delta < \lambda$, let $\dot c_{\gamma,\delta}$ be the $\ka$-Cohen function constructed from the $\ka$-Miller functions
$\dot m_\gamma$ and $\dot m_\delta$. Let $\dot{f}$ be a $\mathbb{MI}^\calF_{\ka, \lambda}$-name for a function $\dot{f}:2^{<\ka} \to 2^{<\ka}$ with $\sigma \subseteq \dot{f}(\sigma)$, for all $\sigma \in 2^{<\ka} $. Assume that a condition $p$ forces that $\dot f$ is
already added by the subforcing $\mathbb{MI}^\calF_{\ka, B}$ for some $B \sub \lambda$. Also assume $\gamma, \delta \notin B$. 
Then, as in the proof of Proposition~\ref{Cohen-model}, we can show that $p \forces \dot{c}_{\gamma,\delta} \notin A_{\dot{f}}$ where
$A_{\dot f}$ is defined as in Subsection 4.1.

Let $\mu < \lambda$ and let $(\dot{f}_\beta : \beta < \mu)$ be $\mathbb{MI}^\calF_{\ka, \lambda}$-names for functions from $2^{<\ka}$ to $2^{<\ka}$ with $\sigma \subseteq \dot{f}_\beta(\sigma)$ for all $\sigma \in 2^{<\ka}$ and $\beta < \mu$. By the $\kappa^{++}$-cc we can find sets $B_\beta \subseteq \lambda$, for all $\beta< \mu$, of size at most $\ka^+$ such that $\dot{f}_\beta$ is added by the subforcing $\mathbb{MI}^\calF_{\ka,  B_\beta}$. Hence, if $(\gamma, \delta) \notin \bigcup_{\beta < \mu} B_\beta$, we have $\forces \dot{c}_{\gamma,\delta} \notin \bigcup_{\beta < \mu} A_{\dot{f}_\beta}$, and $\cov(\M_\ka) \geq \lambda$ follows. 

For $\non(\M_\ka)\leq \ka^+$ it is enough to see that the set of the first $\ka^+$ many $\ka$-Cohen functions $\{ \dot c_{\gamma,\delta} :
\gamma < \delta < \kappa^+ \}$ is non-meager in the generic extension. Fix $\dot f_\beta$, $\beta < \kappa$, as before. 
Also let $p \in \mathbb{MI}^\calF_{\ka, \lambda}$ be arbitrary. By the proof of the previous proposition, we can find $q \leq p$ and
$B \sub \lambda$ of size $\leq \kappa$ such that $q$ forces all $\dot f_\beta$ are already added by the subforcing $\mathbb{MI}^\calF_{\ka, B}$.
Thus $q$ forces that $\dot c_{\gamma,\delta}$ is not contained in the union of the $A_{\dot f_\beta}$, as required.
\end{proof}



\end{document}